\documentclass[preprint,english,11pt,authoryear]{elsarticle}

\usepackage{color}
\usepackage{array}
\usepackage{geometry}
\usepackage[T1]{fontenc}
\usepackage[utf8]{inputenc}
\usepackage{graphicx}
\usepackage{subfigure}
\usepackage{amsmath,amsfonts,amsthm,amssymb}
\usepackage{tikz}
\usepackage{pgfplots}
\usepackage{ulem}
\usepackage{url}
\usepackage{lineno}

\makeatletter
\def\ps@pprintTitle{%
  \let\@oddhead\@empty
  \let\@evenhead\@empty
  \let\@oddfoot\@empty
  \let\@evenfoot\@oddfoot
}
\makeatother

\usepackage{appendix}

\usepackage[noabbrev]{cleveref}
\creflabelformat{equation}{#2(#1)#3}
\crefname{thm}{Theorem}{Theorems}
\crefname{app}{Appendix}{Appendices}

\crefname{prop}{Proposition}{Propositions}
\crefname{lem}{Lemma}{Lemmas}
\creflabelformat{problem}{#2(#1)#3}
\renewcommand{\a}{\boldsymbol \alpha }

\newcommand{\N}{\mathbb{N}}
\newcommand{\p}{\partial}

\newcommand{\R}{\mathbb{R}}

\newcommand{\x}{\mathbf{x}}
\newcommand{\y}{\mathbf y}

\newcommand{\norm}[1]{\left\lVert#1\right\rVert}

\newcommand{\tq}{\tilde q}

\newcommand{\ds}{\displaystyle}
\newcommand{\mb}{\overline{ m}}
\newcommand{\xb}{\overline{ \x}}
\newcommand{\xxb}{\overline{x}_1}

\newcommand{\Oc}{\mathcal  O}
\def\Ui{U}
\def\mutx{\mathcal{M}^*}
\def\mutxa{\mathcal{M}}
\newcommand{\nub}{\vec{\nu}}

\definecolor{aquamarine}{rgb}{0.13, 0.68, 0.8} 
\def\lp {\left( }
\def\rp {\right) }

\def\Rd{\color{black}}
\def\Bk{\color{black}}
\newcommand{\baco}{\left\{ \begin{array}}
	\newcommand{\eaco}{\end{array} \right.}

\newtheorem{theorem}{Theorem}[section]
\newtheorem{proposition}[theorem]{Proposition}

\newtheorem{remark}[theorem]{Remark}

\theoremstyle{definition}

\newtheorem{lem}[theorem]{Lemma}

\usepackage{chngcntr}

\begin{document}

\begin{frontmatter}

\title{The emergence of a birth-dependent mutation rate in asexuals: causes and consequences}


\author{F. Patout$^{\hbox{a}}$,
 R. Forien$^{\hbox{\small{ a}}}$, M. Alfaro$^{\hbox{\small{ a},\small{ b}}}$, J. Papa\"ix$^{\hbox{\small{ a}}}$ and L. Roques$^{\hbox{\small{ a,*}}}$}
 
\address{$^{\hbox{a }}$ INRAE, BioSP, 84914, Avignon, France}

\address{$^{\hbox{b }}$ Universit\'e de Rouen Normandie, CNRS, Laboratoire de Math\'ematiques Rapha\"el Salem, Saint-Etienne-du-Rouvray, France}

\address{$*$ \textit{corresponding author: }lionel.roques@inrae.fr}


\begin{abstract}
	 In unicellular organisms such as bacteria and in most viruses, mutations mainly occur during reproduction. Thus, genotypes with a high birth rate should have a higher mutation rate. However, standard models of asexual adaptation such as the `replicator-mutator equation' often neglect this  generation-time effect.  In this study, we investigate the emergence of a positive dependence between the birth rate and the mutation rate in models of asexual adaptation and the consequences of this dependence.
We show that it emerges naturally at the population scale, based on a large population limit of a stochastic time-continuous individual-based model 
with elementary assumptions. We derive a reaction-diffusion framework that describes the evolutionary trajectories and steady states in the presence of this dependence. 
When this model is coupled with a phenotype to fitness landscape with two optima, one for birth, the other one for survival, a new trade-off arises in the population. Compared to the standard approach with a constant mutation rate, the symmetry between birth and survival is broken. Our analytical results and numerical simulations show that the trajectories of mean phenotype, mean fitness and the stationary phenotype distribution are in sharp contrast with those displayed for the standard model.  The reason for this is that the usual weak selection limit does not hold in a complex landscape with several optima associated with different values of the birth rate. 
Here, we obtain trajectories of adaptation where the mean phenotype of the population is initially attracted by the birth optimum, but eventually converges to the survival optimum, following a hook-shaped curve which illustrates the antagonistic effects of mutation on adaptation. 
	\end{abstract}


\end{frontmatter}

	\section{Introduction}
The effect of the mutation rate on the dynamics of adaptation is well-documented, both experimentally \citep[e.g.,][]{GirMat01,AndDai04} and theoretically. Regarding  theoretical work, since the first studies  on the accumulation of mutation load  \citep{Hal37,KimMar66}, several modelling approaches have investigated the effect of the mutation rate on various aspects of the adaptation of asexuals. This includes lethal mutagenesis theory \citep{BulSan07,BulWil08}, where too high mutation rates may lead to extinction, evolutionary rescue \citep{AncLam19} or the invasion of a sink \citep{LavMar20}. The evolution of the mutation rate per se is also the subject of several models \citep{andre2006evolution,Lyn10}.

The fact that mutation rates per unit time should be higher in species with a shorter generation, given a fixed mutation rate per generation, is called the \textit{generation-time effect}, and has been discussed by \cite{Gil91}. The within-species consequences of the generation-time effect have attracted less attention.  For unicellular organisms such as bacteria, mutations occur during reproduction by means of binary fission \citep{Van98,TruTre09}, meaning that individuals with a high birth rate should have a higher mutation rate (they produce more mutant offspring per unit of time).  This is also true for viruses, as mutations mostly arise during replication \citep{SanDom16}. The probability of mutation during the replication is even greater in RNA viruses as their polymerase lacks the  proofreading activity found in the polymerase of DNA viruses \citep{LauFry13}. As some cancer studies emphasise, with the observation of dose-dependent mutation rates \citep{liu2015dose}, the mutation rate of cancer cells at the population scale can  also be correlated with the reproductive success, through the individual birth rate. On the other hand,  most models that describe the dynamics of adaptation of asexual phenotypically structured populations assume a constant mutation rate across phenotypes \citep[e.g.,][]{GerCol07,SniGer10,DesFis11,AlfCar14,GanMir17,GilHamMarRoq19}. Variations in the individual mutation rate per generation can be caused by genotypic variability \citep{ShaAgr12}, environmental factors  \citep{hoffmann2000environmental} or more generally `G x E' interactions. 
The above-mentioned modelling approaches ignore 
these processes but do take into account a certain variability in the reproductive success.  The main goals of the current study is to determine in which context the  generation-time effect should be taken into account in these models  and to understand the consequences of such birth rate - mutation rate dependence on the evolutionary trajectory of the population.

These consequences are not easy to anticipate as the birth rate is also involved in trade-offs with other life-history traits. Such trade-offs play a crucial role in shaping evolution \citep{stearns89}. They create evolutionary compromises, for instance  between dispersal and reproduction \citep{nathan2001challenges,smith2014programmed,helms2015reproduction,xiao2015seed} or between the traits related to survival and those related to birth \citep{taylor1991optimal}. 
In this last case, we expect that the consequences of the trade-off on the dynamics of adaptation strongly depend on the existence of a positive correlation between the birth rate and the mutation rate.  High mutation rates tend to promote adaptation when the population is far from equilibrium \citep{sniegowski2000evolution} 
but eventually have a detrimental effect due to a higher mutation load   when it approaches a mutation-selection equilibrium \citep{AncLam19}. This ambivalent effect of mutation may therefore lead to complex trajectories of adaptation when the birth and mutation rates are correlated.

In the classical models describing the dynamics of adaptation of a phenotypically structured population, the breeding values at a set of $n$ traits are described by a vector $\x \in \Omega \subset \R^n$. 
The breeding value for a phenotypic trait is usually defined as the total additive effect of its genes on that trait, see \citep{FalMac96,Kru04} and is independent of the environmental conditions, given the genotype. For simplicity and consistency with other modelling studies, we will call $\x$ the `phenotype' in the following, although it still represents breeding values.  
The effect of mutations on the phenotype distribution is described through a linear operator $\mutxa$ which does not depend on the parent phenotype $\x$. The operator $\mutxa$  can be described with a convolution product involving a mutation kernel \citep{ChaFer06,GilHamMarRoq17} or with a Laplace operator \citep{Kim64,Lan75,AlfCar14,HamLav20}, corresponding to a diffusion approximation of the mutation effects. Under the diffusion approximation,  $\mutxa(\cdot)=D\, \Delta (\cdot)$ with $D>0$ a constant coefficient which is proportional to the mutation rate and to the mutational variance at each trait. 
The Malthusian fitness $m(\x)$, i.e., the Malthusian growth rate of individuals with phenotype $\x$, is defined as the difference between the birth rate $b(\x)$ and death rate $d(\x)$ of this class of individuals:
\begin{equation}\label{bhetero}
	m(\x)= b(\x)-d(\x).
\end{equation}
The following generic equation then describes the combined effects of mutation and selection on the dynamics of the phenotype density $f(t,\x)$ under a diffusive approximation of the mutation effects:
		\begin{equation}\label{eq:modUcst}
	\p_t f(t,\x) = D\, \Delta f(t,\x) + f(t,\x) m(\x),
	\end{equation}
	in the absence of density-dependent competition, or 
		\begin{equation}\label{eq:f_dens_dep}
	\p_t f(t,\x) = D\, \Delta f(t,\x) + f(t,\x) \lp m(\x)-\int_\Omega f(t,\y) \, d\y\rp,
	\end{equation}
	if density-dependent competition is taken into account. In both cases, the equation satisfied by the frequency $q(t,\x)=f(t,\x)/N(t)$ (with $N(t) = \int_\Omega f(t, \x) d\x$ the total population size) is
	\begin{equation} \label{eq:q} \tag{$\mathcal{Q}_{stand}$}
	    \partial_t q(t,\x) = D\, \Delta q(t,\x) + q(t,\x) (m(\x) - \mb(t)), \ t> 0, \ \x\in \Omega\subset \R^n,
	\end{equation}
with $\mb(t)$ the mean fitness in the population:
\begin{equation} \label{def:mb}
   \mb(t) = \int_\Omega m(\x)  \,q(t,\x) \, d\x.
\end{equation}
These models allowed a broad range of results in various biological contexts: concentration around specific traits \citep{DieJab05,LorMir11,MarRoq16}; explicit solutions \citep{AlfCar14,Bik14,AlfCar17}; moving and/or fluctuating optimum \citep{FigMir19,RoqPat20}; anisotropic mutation effects \citep{HamLav20}. Then can aslo be extended in order to take migration events into account \citep{DebRon13,LavMar20}.

 With these models, the dynamics of adaptation and the equilibria only depend on the birth and death rates through their difference $m(\x)= b(\x)-d(\x).$ Thus, these models do  not discriminate between phenotypes for which both birth and death rates are high compared to those for which they are both low, given that the difference is constant. However, as explained above, the mutation rate may be positively correlated with the birth rate  which could generate an imbalance in favour of one of the two strategies: having a high birth rate vs. having a high survival rate.  To acknowledge the role of phenotype-dependent birth rate and the resulting asymmetric effects of fertility and survival in a deterministic setting, a new paradigm is necessary.

In this work, we consider the case of mutations that occur during the reproduction of asexual organisms. We assume that the probability of mutation per   birth event $\Ui$ \textit{does not} depend on the phenotype of the parent. On the other hand, following classical adaptive landscape approaches \citep{Ten14}, the birth and death rates do depend on the phenotype. Using these basic assumptions, we consider in Section~\ref{sec:causes} standard  stochastic individual-based models of adaptation with mutation and selection.
We present how the standard model \eqref{eq:q} appears naturally as a large population limit of both a discrete-time model and a continuous-time model when the variance of mutation effects is small and when selection is weak, i.e., when the variations of the birth and death rates across the phenotype space are very small.
In this work, however, we are interested in a particular setting where this assumption is not satisfied.
In this case, using results from \cite{fournier_microscopic_2004} for the continuous-time model, we argue that, when the mutation variance is small, a more accurate approximation of the mutation operator is given by $\mutxa (q) = D \Delta (b(\cdot)\, q)$, leading to a new equation of the form
\begin{equation}\label{eq:main_model} \tag{$\mathcal{Q}_b$}
	    \p_t q(t,\x) = D \Delta( b\, q)(t,\x) + q(t,\x) (m(\x) - \mb(t)), \  t> 0, \ \x\in \Omega\subset \R^n.
\end{equation}
Here, the mutation operator $D \Delta( b(\x)\, q)(t,\x)$ depends on the phenotype $\x$ through the birth rate $b(\x)$, translating the fact that new mutants appear at a higher rate when the birth rate increases.
 Models comparable to \eqref{eq:main_model} (but with a discrete phenotype space) appear in the literature, and lead in some cases to quite similar results as the standard model \eqref{eq:q}  \citep{Hof85,BaaGab00}. However, \Rd this is not always the case, \Bk as shown in this contribution.

In Section~\ref{sec:pde}, we use \eqref{eq:main_model} to study the evolution of the phenotype distribution when the population is subjected to a trade-off between a birth optimum and a survival optimum, and we highlight the main differences with the standard approach \eqref{eq:q}.
More specifically, we  study the evolution of the phenotype distribution in the presence of a fitness optimum where $b$ and $d$ are both large (the birth or reproduction optimum), and a survival optimum, where $b$ and $d$ are both small, but such that the difference $b-d$ is symmetrical.

Based on analytical results and numerical simulations, we compare the trajectories of adaptation and the equilibrium phenotype distributions between these two approaches and we check their consistency with the underlying individual-based models. We discuss these results in \Cref{sec conc}.

\bigskip

\section{Emergence of a birth-dependent mutation rate in an individual-based setting \label{sec:causes}}

In this section, we present how the standard equation \eqref{eq:q} and the new model \eqref{eq:main_model} with birth-dependent mutation rate are obtained from large population limits of stochastic individual-based models.
We first state a convergence result due to \cite{fournier_microscopic_2004} which provides the convergence of the phenotype distribution of the population to the solution of an integro-differential equation, when the size of the population tends to infinity.
We then show that, when the variance of the mutation effects is small, this equation yields the new model \eqref{eq:main_model}.
This shows how a dependence between the birth rate and the rate at which new mutant appear in the population arises, even though the probability of mutation per birth event $\Ui$ does not depend on the phenotype of the parent.
We then treat the case of weak selection, and show how the model \eqref{eq:q} is obtained as a large population limit of the phenotype distribution with a specific time scaling, using results in \cite{ChaFer08}.
We also state an analogous result for the discrete-time model, where in the same regime of weak selection and small mutation effects, we show the convergence to the solution of \eqref{eq:q} as the population size tends to infinity, on the same timescale as the other model.

%
%

In this individual-based setting, we consider a finite population of size $N_t$ where each individual carries a phenotype in a bounded open set $\Omega \subset \R^n$. If the individuals at time $t$ have phenotypes $\lbrace \x_1, \ldots, \x_{N_t}\rbrace$, we record the state of the population through the empirical measure
\begin{equation*}
    \nu^K_t = \frac{1}{K} \sum_{i=1}^{N_t} \delta_{\x_i}, \qquad t \geq 0,
\end{equation*}
where $\delta_{\x}$ is the Dirac measure at the point $\x \in \Omega$.  Note that the number of individuals $N_t$ in this stochastic individual-based setting does not correspond to the quantity $N(t)$ defined in the introduction. In fact, these two quantities will be related via the scaling parameter $K > 0$: $N(t) = \lim_{K \to \infty} N_t/K$ (or $N(t) = \lim_{K \to \infty} N_{t/\varepsilon_K}/K$ if time is rescaled, see the definition of $\varepsilon_K$ below).

Let $M_F(\Omega)$ denote the space of finite measures on $\Omega$, endowed with the topology of weak convergence.
For any $\nu \in M_F(\Omega)$ and any measurable
and bounded function $ \phi : \Omega \to \R$, we shall write
\begin{equation*}
    \langle \nu, \phi \rangle = \int_\Omega \phi(\x) \nu(d\x).
\end{equation*}

\subsection{Derivation of the model \eqref{eq:main_model} with birth-dependent mutation rate}

We first consider a continuous-time stochastic individual-based model where individuals die and reproduce at random times depending on their phenotype and the current population size.
We let $b:\Omega \to \R_+$ and $d:\Omega \to \R_+$ be two bounded and measurable functions, and we assume that an individual with phenotype $\x\in\Omega$ reproduces at rate $b(\x)$ and dies at rate $d(\x) + c_K N_t$ for some $c_K>0$. 
This parameter $c_K$ measures the intensity of competition between the individuals in the population, and prevents the population size from growing indefinitely.
Each newborn individual either carries the phenotype of its parent, with probability $1-U$, or, with probability $U$, carries a phenotype $\y$ chosen at random from some distribution $\rho(\x,\y)d\y$, where $\x$ is the phenotype of its parent.

We can now describe the limiting behaviour of this model when the parameter $K$ tends to infinity.
The following convergence result can be found for example in \citet[Theorem~5.3]{fournier_microscopic_2004} and \citet[Theorem~4.2]{ChaFer08}.
Let $D([0,T],M_F(\Omega))$ denote the Skorokhod space of c\`adl\`ag functions taking values in $M_F(\Omega)$.

\begin{proposition}  \label{prop:cvg_overlapping}
    Assume that $\nu^K_0$ converges weakly to a deterministic $f_0 \in M_F(\Omega)$ as $K \to +\infty$ and that $ c_K = c/K$     for some $c>0$.
    Also assume that $\sup_K \mathbb{E}[\langle \nu^K_0, 1 \rangle^3] < +\infty$.
    Then, for any fixed $T > 0$, as $K \to +\infty$,
    \begin{equation*}
        \left( \nu^K_t, t \in [0,T] \right) \longrightarrow \left( f_t, t \in [0,T] \right),
    \end{equation*}
    in distribution in $D([0,T],M_F(\Omega))$, where $(f_t, t \in [0,T])$ \Rd is the unique deterministic function taking values in $M_F(\Omega)$ such that, \Bk for any bounded and measurable $\phi : \Omega \to \R$,
    \begin{equation}\label{eq:q_overlap}
        \langle f_t, \phi \rangle = \langle f_0, \phi \rangle + \int_0^t \langle f_s, b\, \mutx\phi + (b-d-c \langle f_s, 1 \rangle) \phi\rangle ds,
    \end{equation}
    where
    \begin{equation*}
        \mutx\phi(\x) = U \int_\Omega (\phi(\y)-\phi(\x)) \rho(\x,\y) d\y.
    \end{equation*}
\end{proposition}

We note that, if $f_0$ {\Rd is absolutely continuous} with respect to the Lebesgue measure, $f_t$ admits a density (denoted by $f(t,\cdot)$) for all $t \geq 0$.
In this case, setting $m(\x) = b(\x) - d(\x)$ and
\begin{linenomath}
\begin{align*}
    q(t,\x) = \frac{f(t,\x)}{\langle f_t, 1 \rangle}, && \mb(t) = \int_\Omega q(t,\x) m(\x) d\x,
\end{align*}
\end{linenomath}
we see that the phenotype distribution $q$ solves the following
\begin{equation} \label{eq:q_non_local}
    \partial_t q(t,\x) = \mutxa(b\,q)(t,\x) + \left(m(\x) - \mb(t)\right) q(t,\x),
\end{equation}
where
\begin{equation*}
    \mutxa(f)(\x) = U \left(\int_\Omega f(\y) \rho(\y,\x) d\y - f(\x)\right).
\end{equation*}
\Rd Uniqueness of a solution to \eqref{eq:q_overlap} is proved in \cite{fournier_microscopic_2004} (Theorem~5.3). \Bk
In the model~\eqref{eq:q_non_local},  due to the coefficient $b$ in $\mutxa(b\,q)$, \Rd mutation occurs at a higher rate \Bk in regions where $b$ is higher. As we shall \Rd explain below, this has far reaching \Bk consequences on the qualitative behaviour of the phenotype distribution, which the standard model \eqref{eq:q} does not capture.
 However, the analysis of the integro-differential equation \eqref{eq:q_non_local} is very intricate.
If the variance of mutation effects is sufficiently small, we can instead study a diffusive approximation of equation \eqref{eq:q_non_local}.
Assume that the effects of mutation on phenotype can be described by a mutation kernel $J$, such that $\rho(\y,\x) = J(\x-\y)$. Namely,
\begin{equation*}
    \mutxa(b\,f)(\x) = U \,\lp \int_{\Omega} (b\, f)(\x-\y) \, J(\y) \, d\y - (b\, f)(\x)\rp.
\end{equation*}
Formally, we write a Taylor expansion of $ (b \, q)(t,\x-\y)$ at $\x\in \Omega$:
\[(b \, q)(t,\x-\y)=\sum_{k_1,\ldots,k_n=0}^\infty(-1)^{k_1+\cdots+k_n}\,\frac{y_1^{k_1}\cdots y_n^{k_n}}{k_1 ! \, \cdots \,  k_n !}\,\frac{\partial^{k_1+\cdots+k_n}(b\,q)}{\partial x_1^{k_1}\cdots \partial x_n^{k_n}}(t,\x).\]
We define the central moments of the distribution:
\[\omega_{k_1,\ldots, k_n}=\int_{\R^n} y_1^{k_1}\cdots y_n^{k_n} \, J(y_1,\dots, y_n)\, dy_1 \dots dy_n.\]We make a symmetry assumption on the kernel $J$ which implies that $\omega_{k_1,\ldots, k_n}=0$ if at least one of the $k_i$'s is odd. Moreover, we assume the same variance $\lambda$ at each trait: $\omega_{0,\ldots, 0,k_i=2,0, \ldots, 0 }=\lambda$, and that the moments of order $k_1+\cdots+k_n\ge 4$ are of order $O(\lambda^2)$. These assumptions are satisfied with the classic isotropic Gaussian distribution of mutation effects on phenotype. For $\lambda \ll 1$, we obtain:
\[U \,\lp \int_{\Omega} (b\, q)(t,\x-\y) \, J(\y) \, d\y - (b\, q)(t,\x)\rp\approx \frac{\lambda \, U}{2} \Delta (b \, q) (t,\x)+O(\lambda^2).\]Thus, when the variance $\lambda$ of the (symmetric) mutation kernel $J$ is small, we expect that the solution to \eqref{eq:q_non_local} behaves as the solution to \eqref{eq:main_model}:
\begin{equation*}
    \partial_t q(t,\x) = D \Delta (b\, q) (t,\x) + (m(\x)-\mb(t))\, q(t,\x),
\end{equation*}
where $D = \lambda \, U/2$. 
\Rd Note that the proof of uniqueness for \eqref{eq:q_overlap} in \cite{fournier_microscopic_2004} can be adapted to prove that this equation admits a unique solution  \citep[see also Theorem~4.3 in][]{ChaFer08}. \Bk

\begin{remark}
We recall that the assumption here is that mutations occur during reproduction (e.g. in unicellular organisms or viruses). If we had assumed that mutations take place at a constant rate during each individual's lifetime, instead of linking them to reproduction events, we would have obtained a different equation in \eqref{eq:q_overlap} leading to the standard model \eqref{eq:q} instead of \eqref{eq:main_model}.
\end{remark}

\subsection{Derivations of the standard model \eqref{eq:q}}

The standard model \eqref{eq:q} is classically derived by letting the variance of the mutation kernel tend to zero and by rescaling time to compensate for the fact that mutations have very small effects.
In order to obtain the convergence of the process $(\nu^K_t, t \geq 0)$ in this regime, one also has to assume that the intensity of selection (measured by $b-d$) is of the same order of magnitude as the variance of the mutation kernel.
This corresponds to a weak selection regime, where $b$ and $d$ are almost constant on $\Omega$.

\paragraph*{Large population limit of the continuous-time model in rescaled timescale}

We consider the same stochastic individual-based model as above, but we allow $b$, $d$ and $\rho$ to depend on $K$.
We thus let $b_K(\x)$ denote the birth rate of individuals with phenotype $\x$, $d_K$ their death rate, and $\rho_K$ will be the mutation kernel.
We then make the following assumption.

\subparagraph{Assumption (SE) (frequent mutations with small effects).}
Let $\varepsilon_K = K^{-\eta}$ for some $0 < \eta < 1$ and assume that $\rho_K$ is a symmetric kernel such that, for all $1 \leq i \leq n$,
\begin{linenomath}
\begin{align*}
     \int_\Omega (y_i-x_i)^2 \rho_K(\x,\y) d\y = \varepsilon_K\, \lambda + o(\varepsilon_K), && \int_\Omega (y_i-x_i)^{2+\delta} \Rd{} \rho_K(\x,\y) \Bk{} d\y = o(\varepsilon_K),
\end{align*}
\end{linenomath}
for all in $\x \in \Omega$, some $\lambda > 0$ and $\delta \in ]0,2]$.

~~

This assumption is what justifies the so-called diffusive approximation, where the effect of mutations on the phenotype density is modelled by a Laplacian in continuous-time.

\subparagraph*{Assumption (WS) (weak selection).} Assume that
\begin{linenomath}
\begin{align*}
    b_K(\x) = 1 + \varepsilon_K\, b(\x), && d_K(\x) = 1 + \varepsilon_K \, d(\x), && c_K = \frac{\varepsilon_K}{K} c, 
\end{align*}
\end{linenomath}
for some bounded functions $b:\Omega \to \R$, $d:\Omega \to \R$ and some positive $c$.

~~

The following result then corresponds to Theorem~4.3 in \cite{ChaFer08}.
Recall that $\Omega$ is assumed to be a bounded open set, and further assume that it has a smooth boundary $\partial \Omega$.
Let $C^2_0(\Omega)$ be the set of twice continuously differentiable functions $\phi : \bar{\Omega} \to \R$ such that
\begin{equation*}
    \nabla \phi(\x) \cdot \nub(\x) = 0, \quad \forall \x \in \partial \Omega,
\end{equation*}
where $\nub(\x)$ is the outward unit normal to $\partial\Omega$.

\begin{proposition} \label{prop:cvg_overlapping_rescaled}
    Let Assumptions (SE) and (WS) be satisfied.
    Also assume that $\nu_0^K$ converges weakly to a deterministic $f_0 \in M_F(\Omega)$ as $K \to \infty$ and that \[\sup_K \mathbb{E}[\langle \nu_0^K, 1\rangle^3] < + \infty.\]
    Then, for any fixed $T > 0$, as $K \to + \infty$,
    \begin{equation*}
        \left( \nu^K_{t/\varepsilon_K}, t \in [0,T] \right) \longrightarrow \left( f_t, t \in [0,T] \right),
    \end{equation*}
    in distribution in $D([0,T],M_F(\Omega))$, where $(f_t,t\in[0,T])$ is \Rd the unique deterministic function taking values in $M_F(\Omega)$ such that\Bk, for any $\phi \in C^2_0(\Omega)$,
    \begin{equation*}
        \langle f_t, \phi\rangle = \langle f_0, \phi \rangle + \int_0^t \langle f_s, D \Delta \phi + (b-d-c \langle f_s, 1 \rangle) \phi \rangle ds,
    \end{equation*}
    with $D = \lambda\, U/2$.
\end{proposition}
For all $t > 0$, if $f_0$ admits a density with respect to the Lebesgue measure, $f_t$ admits a density $f(t,\cdot) \in L^1(\Omega)$ and the phenotype distribution $q(t,\x) = f(t,\x)/\langle f_t, 1 \rangle$ solves \eqref{eq:q}:
\begin{equation*} 
    \partial_t q(t,\x) = D \Delta q(t,\x) + (m(\x) - \mb(t)) \, q(t,\x).
\end{equation*}
As we can see, we have lost the factor $b$ in the mutation term by taking this limit.
This comes from Assumption (WS) which states that $b_K(\x) = 1 + O(\varepsilon_K)$.
As a result this equation does not distinguish the birth optimum from the survival optimum (see Section~\ref{sec:pde}).

\paragraph*{Large population limit of an individual-based model with non-overlapping generations}

We now consider a model where generations are non-overlapping, meaning that, between two generations (denoted $t$ and $t+1$), all the individuals alive at time $t$ first produce a random number of offspring and then die.
The population at time $t+1$ is thus only comprised of the offspring of the individuals alive at time $t$.

Let $w_K : \Omega \to \R_+$ be a measurable and bounded function and assume that an individual with phenotype $\x \in \Omega$ produces a random number of offspring which follows a Poisson distribution with parameter $w_K(\x)$.
In order to include competition, we assume that each of these offspring survives with probability $e^{-c_K N_t}$ for some $c_K > 0$, where $N_t$ is the number of individuals in generation $t$.
Each newborn individual either carries the phenotype of its parent, with probability $1-U$, or, with probability $U$, carries a phenotype $\y$ chosen at random from some distribution $\rho_K(\x,\y)d\y$, where $\x$ is the phenotype of its parent.

We now make several assumptions in order to obtain an approximation of the process as the population size tends to infinity.
For the limiting process to be continuous in time, we need to assume that the change in the composition of the population from one generation to the next is very small, and then rescale time by the appropriate factor.
This ties our hands somewhat, and we need to assume that $w_K$ is very close to one everywhere in $\Omega$.
More precisely, we make the following assumption.

\subparagraph*{Assumption (WS').}  Let $\varepsilon_K = K^{-\eta}$ for some $0 < \eta < 1$ and assume that
\begin{linenomath}
\begin{align*}
    w_K(\x) = \exp\left( \varepsilon_K\, m(\x) \right), && c_K = \frac{\varepsilon_K}{K} c, 
\end{align*}
\end{linenomath}
for some bounded function $m : \Omega \to \R$ and some positive $c$.

~~

Here, $w_K(\x)$ corresponds to the Darwinian fitness (the average number of offspring of an individual with phenotype $\x$), while $m(\x)$ corresponds to the Malthusian fitness (i.e., the growth rate of the population of individuals with phenotype $\x$).
We further assume that $\rho_K$ satisfies Assumption (SE) above.

~~

The large population limit of this process is then given by the following result, which is analogous to similar results in continuous-time (for example in \cite{ChaFer08}). For the sake of completeness, we give its proof in Appendix~\ref{proof firstapp}.
\begin{proposition}  \label{prop:cvg_non_overlapping}
    Assume that Assumption (WS') is satisfied, along with (SE).
    Also assume that, $\nu^K_0$ converges weakly to a deterministic $f_0 \in M_F(\Omega)$.
    Then, for any fixed $T > 0$, as $K \to +\infty$,
    \begin{equation*}
        \left(\nu^K_{\lfloor t / \varepsilon_K \rfloor}, t \in [0,T] \right) \longrightarrow \left( f_t, t \in [0,T] \right),
    \end{equation*}
    in distribution in $D([0,T],M_F(\Omega))$, where $(f_t, t \in [0,T])$ is \Rd the unique deterministic function taking values in $M_F(\Omega)$ such that\Bk, for any $\phi \in C^2_0(\Omega)$,
        \begin{equation} \label{limit_f_discrete}
            \langle f_t, \phi \rangle = \langle f_0, \phi \rangle + \int_0^t \langle f_s, \Rd D \Delta \phi \Bk + (m - c\, \langle f_s, 1 \rangle) \phi \rangle\, ds,
        \end{equation}
    \Rd with $D = \lambda \, U/2$.\Bk
\end{proposition}
For all $t>0$,  if $f_0$ admits a density with respect to the Lebesgue measure, $f_t$ admits a density $f(t,\cdot) \in L^1(\Omega)$.
Then $f(t,\cdot)$ solves the equation
\begin{equation*}
    \partial_t f(t,\x) = D \Delta f (t,\x) + \left( m(\x) - c \int_\Omega f(t,\y) d\y \right) f(t,\x).
\end{equation*}
We also note that the phenotype distribution $q(t,\x) = f(t,\x)/\langle f_t, 1 \rangle$ solves 
 \eqref{eq:q}.

~~
Propositions~\ref{prop:cvg_overlapping_rescaled} and \ref{prop:cvg_non_overlapping} show how the standard model \eqref{eq:q} arises as a large population limit of individual-based models in the weak selection regime with small mutation effects.
However, as Proposition~\ref{prop:cvg_overlapping} shows, the fact that the birth rate does not appear in the mutation term is a consequence of the weak selection assumption.
In the next section, we will focus on a situation corresponding to a strong trade-off between birth and survival. In this case, the weak selection assumption is not satisfied. Thus, the new model \eqref{eq:main_model} should be more appropriate to study the dynamics of adaptation, at least when generations are overlapping.

In the model with non-overlapping generations, we expect that the model \eqref{eq:q} emerges even when the weak selection assumption is not satisfied. From an intuitive perspective, with this model, the expected number of mutants per generation is $\Ui \, N(t)$. Thus, if $N(t)$ is close to the carrying capacity, the overall number of mutants should not depend on the phenotype distribution in the population. However, if one tries to take a large population limit of the discrete-time model in the same regime as in Proposition~\ref{prop:cvg_overlapping} (keeping $w$ and $\rho$ fixed and letting the population size tend to infinity), then the phenotype distribution converges to the solution to a deterministic recurrence equation of the form
\begin{equation*}
    \langle q_{t+1}, \phi\rangle = \left\langle q_t, \tfrac{w}{\overline{w}(t)} \mutx\phi + \tfrac{w}{\overline{w}(t)} \phi \right\rangle,
\end{equation*}
where $\mutx$ is as in \eqref{eq:q_overlap}.
We do not study this equation here, but it is interesting to note that the fitness has an effect on the mutations, albeit quite different from that in \eqref{eq:q_overlap}.

In the following section, we use \eqref{eq:main_model} to study the consequences of a birth-dependent mutation rate on the trade-off between birth and survival, and we compare our results to the standard approach of \eqref{eq:q} and to individual-based simulations.

\section{Consequences of a birth-dependent mutation rate on the trade-off between birth and death}\label{sec:pde}

We focus here on the trajectories of adaptation and the large time dynamics given by the model~\eqref{eq:main_model}, with a special attention on the differences with the standard approach~\eqref{eq:q} which neglects the dependency of mutation rate on  birth rate. 

In most related studies, the relationships between the phenotype $\x$ and the fitness $m(\x)$ is described with the standard Fisher's Geometrical Model (FGM) where $m(\x) = r_{max}- \norm{\x}^2/2$. This phenotype to fitness landscape model is widely used, see e.g. \cite{Ten14}, \cite{MarLen15}. It has shown robust accuracy to predict distributions of pathogens \citep{MarLen06a,MarEle07}, and to fit biological data \citep{Perthe14,SchHwa16}.
	Here, however, in order to study the trade-off between birth and survival, we shall assume that the death rate $d$ takes the form: $d(\x) = r-s(\x)$ for some $r>0$, such that
	\begin{equation}\label{eq:def_m=r+b+s}
	    m(\x)=\underbrace{b(\x)}_\text{birth}+\underbrace{s(\x)}_\text{survival}-r,
	\end{equation}
	for some function $s:\Omega \to [0,r]$ such that $b$ and $s$ are symmetric about the axis $x_1=0$, in the sense of \eqref{eq:sym_b_s}, and we assume that $b$ has a global maximum that is not on this axis. 
	As a result $s$ also has a global maximum, which is the symmetric of that of $b$. 
The positive constant $r$ has no impact on the dynamics of the phenotype distribution $q(t,\x)$ in model \eqref{eq:main_model}, as it vanishes in the term $m(\x)-\mb(t)$. To keep the model relevant, the constant $r$ must therefore be chosen such that $d(\x)>0$ for all $x\in \Omega$.

We assume that $b(\x)$ reaches its maximum at  $\Oc_b\in \Omega$ and $s(\x)$ reaches its maximum at $\Oc_s\in \Omega$. If one of the optima leads to a higher fitness value, we expect that the corresponding strategy (high birth vs. high survival) will be selected. To avoid such `trivial' effects, and to analyse the result of the trade-off between birth and survival independently of any fitness bias towards one or the other, we make the following assumptions. The domain $\Omega$ is symmetric about the hyperplane $\{x_1=0\}$. Next, $b$ and $s$ are positive, continuous over $\overline{\Omega}$ and symmetric in the following sense:
\begin{equation} \label{eq:sym_b_s}
  b(\x)=s(\iota(\x)), \hbox{ with }\iota(\x)=\iota(x_1,x_2, ..., x_n) = (-x_1,x_2, ..., x_n).  
\end{equation}
The optima are then also symmetric about the axis $x_1=0$:
\[\Oc_b=(\beta,0,\ldots,0) \hbox{ and }\Oc_s=(-\beta,0,\ldots,0),\]for some $\beta>0$, so that the birth optimum is situated to the right of $x_1=0$ and the survival optimum is situated to the left of $x_1=0$. A schematic representation of the birth and survival terms and corresponding fitness function, along the first dimension $x_1$ is given in Fig.~\ref{fig:camel}.

\begin{figure}
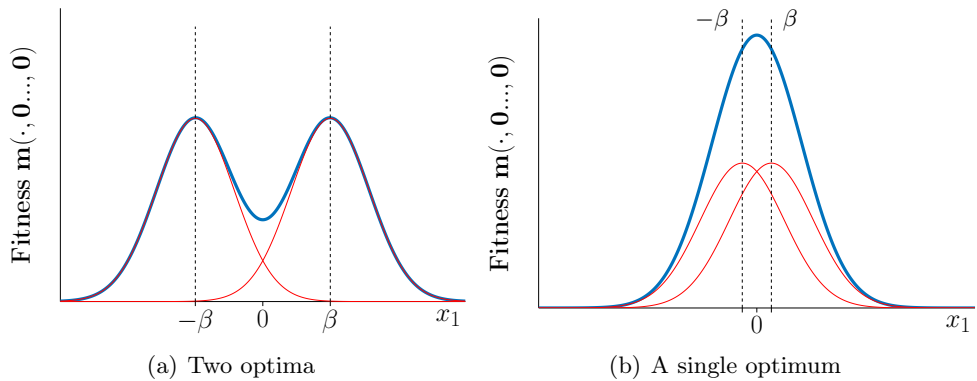

	\center
\subfigure[Two optima]{	\input{figfitness2.tex}}
\subfigure[A single optimum]{	\input{noptim.tex}}
	\caption{{\bf Schematic representation of the fitness function $m(\x)$ along the phenotype dimension $x_1$.} In both cases the black dashed lines correspond to the survival optimum $\Oc_s=(-\beta,\ldots,0)$ (on the left) and the birth optimum $\Oc_b=(\beta,0,\ldots,0)$ (on the right).  In panel (a) those optima are almost superposed with those of $m$, which is very different from panel (b). In red we pictured the functions $b-r/2$ and  $s-r/2$. }
	\label{fig:camel}
\end{figure}

Finally, we assume that the birth rate is larger than the survival rate in the whole half-space  around $\Oc_b$ $(\Omega \cap \{x_1> 0\})$, and conversely, from \eqref{eq:sym_b_s}, the survival rate is higher in the other half-space. In other terms:
\begin{equation}\label{hyp_b}
\begin{array}{l}
 b(x_1,\ldots,x_n) > s(x_1,\ldots,x_n), \hbox{ for all }\x\in \Omega \cap \{x_1> 0\},  \\
 s(x_1,\ldots,x_n) > b(x_1,\ldots,x_n), \hbox{ for all }\x\in \Omega \cap \{x_1< 0\}.
\end{array}
\end{equation}
From the symmetry assumption \eqref{eq:sym_b_s}, we know that the hyperplane $\{x_1=0\}$ is a critical point for  $b+s$  in the direction $x_1$, that is $\partial_{x_1}b(0,x_2,\ldots,x_n)=-\partial_{x_1}s(0,x_2,\ldots,x_n)$.  

For the well-posedness of the model~\eqref{eq:main_model}, and as the integral of $q(t,\x)$ over $\Omega$ must remain equal to $1$ (recall that $q(t,\cdot)$ is a probability distribution), we assume reflective (Neumann) boundary conditions:\[b(\x) (\nabla q(t,\x) \cdot \nub(\x))+(\nabla b(\x) \cdot \nub(\x))\, q(t,\x)=0, \x\in \partial\Omega,\]with $\nub(\x)$ the outward unit normal to $\partial \Omega$, the boundary of $\Omega$. We also assume a compactly supported initial condition $q_0(\x)=q(0,\x)$, with integral $1$ over $\Omega$.

\subsection{Trajectories of adaptation \label{sec:traj_adapt}}

The methods developed in \cite{HamLav20} provide analytic formulas describing the full dynamics of adaptation, and in particular the dynamics of the mean fitness $\mb(t)$, for models of the form~\eqref{eq:q}, i.e., with a constant mutation rate. As far as model~\eqref{eq:main_model} is concerned, due to the birth-dependent term in the mutation operator $D\, \Delta(b\, q)$, the derivation of comparable explicit formulas seems out of reach. To circumvent this issue,  we use numerical simulations to exhibit some qualitative properties of the adaptation dynamics, that we demonstrate next. We focus on the dynamics of the mean phenotype $\xb(t)$ and of the mean fitness $\mb(t)$, to be compared to the `standard' case, where the mutation rate does not depend on the phenotype, and to individual-based stochastic simulations with the assumptions of Section \ref{sec:causes}. In the PDE (partial differential equation) setting, the mean phenotype $\xb(t)\in \Omega$ and mean fitness $\mb(t)\in \R$ are defined by:
\[\xb(t):=\int_{\Omega}\x \, q(t,\x) \, d\x, \quad \mb(t):=\int_{\Omega}m(\x) \, q(t,\x) \, d\x.\]

\paragraph{Numerical simulations} Our numerical computations are carried out in dimension $n=2$, starting with an initial phenotype concentrated at some point $\x_0$ in $\Omega$. We solved the PDEs with a method of lines (the Matlab codes are available in the Open Science Framework repository: \url{https://osf.io/g6jub/}). The trajectories given by the PDE \eqref{eq:main_model} with a birth-dependent mutation rate are depicted in Fig.~\ref{fig:traj_bdep}a, together with 10 replicate simulations of a stochastic individual-based model with overlapping generations (see Section~\ref{sec:causes}). The mean phenotype is first attracted by the birth optimum $\Oc_b$. In a second time, it converges towards $\Oc_s$. This pattern leads to a trajectory of mean fitness which exhibits a small `plateau':   the mean fitness seems to stabilise at some value smaller than the ultimate value $\mb_\infty$ during some period of time, before growing again at larger times. The trajectories given by individual-based simulations exhibit the same behaviour. 

On the other hand, simulation of the standard equation~\eqref{eq:q}  without dependence of the mutation rate with respect to the phenotype (with Neumann boundary conditions), leads to standard saturating trajectories of adaptation, see~Fig.~\ref{fig:traj_bdep}b \citep[already observed in][with this model]{MarRoq16}. This time,  the trajectories given by the model \eqref{eq:q} are in good agreement with those given by an individual-based model with non-overlapping generations (see Section~\ref{sec:causes}).

If the initial population density $q_0$ is symmetric about the hyperplane $\{x_1=0\}$, then so does $q(t,\x)$ at all positive times in this case.  We indeed observe that if $q(t,\x)$ is a solution of \eqref{eq:q} with initial condition $q_0$, then so does  $q(t,\iota(\x))$. By uniqueness \citep[which follows from][]{HamLav20},
 $q(t,\x)=q(t,\iota(\x))$ at all times. This in turns implies that the mean phenotype $\xb(t)$ remains on the hyperplane $\{x_1=0\}$, i.e., at the same distance of the two optima $\Oc_b$ and $\Oc_s$. Besides, even if $q_0$ was not symmetric about $\{x_1=0\}$, i.e., if the initial phenotype distribution was biased towards one of the two optima, the trajectory of $\xb(t)$ would ultimately still converge to the axis  $\{x_1=0\}$.  Again, this is a consequence of the uniqueness of the positive stationary state of ~\eqref{eq:q} (with integral $1$), which is itself a consequence of the uniqueness of the principal eigenfunction (up to multiplication) of the operator $\phi\mapsto D\, \Delta\, \phi + m(\x) \,\phi$ \citep[this uniqueness result is classical, see e.g.][]{AlfVer18}.

	\begin{figure}
		\center
		\subfigure[Model \eqref{eq:main_model}]{
		\includegraphics[width=0.32\textwidth]{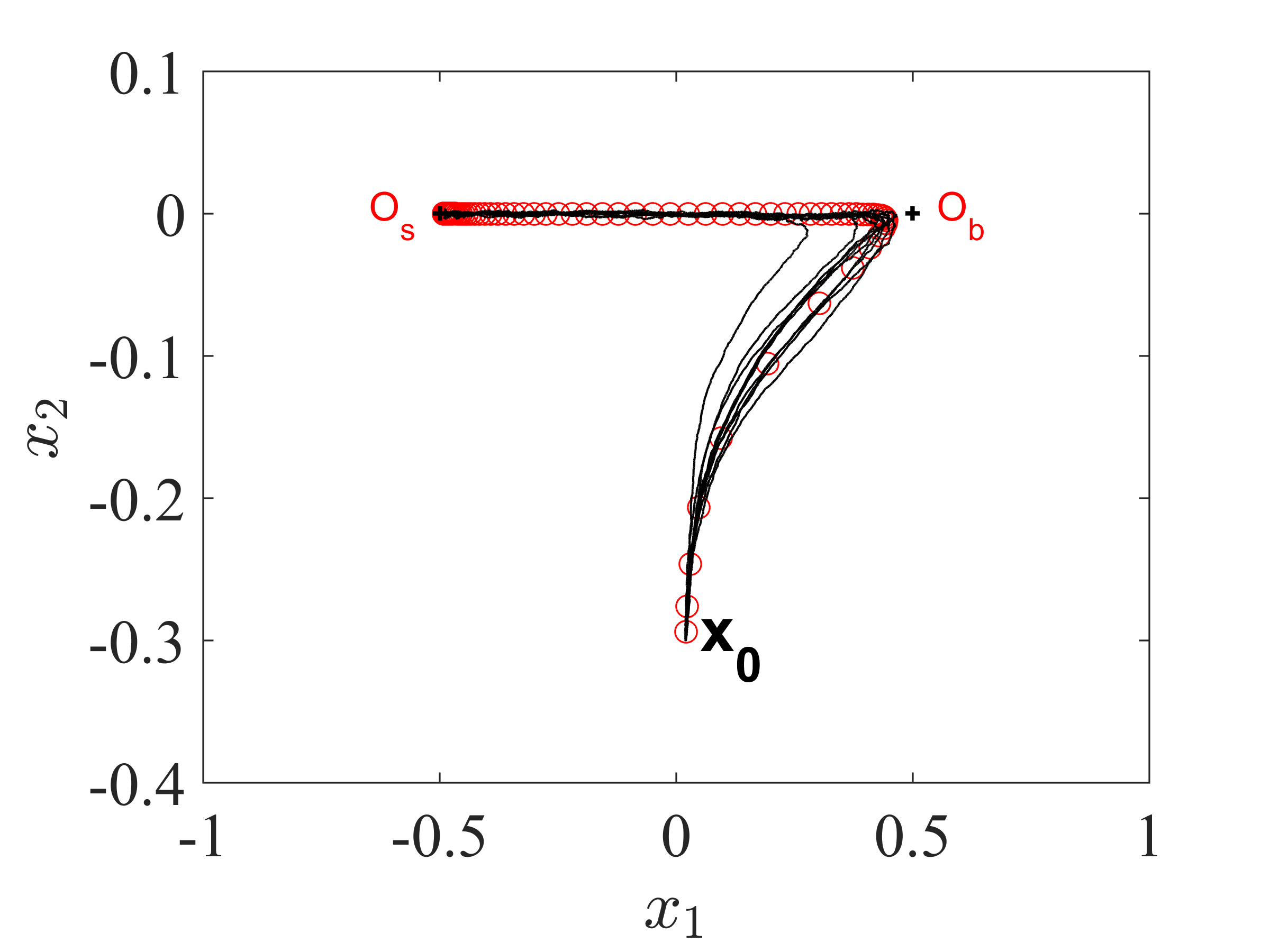}
		\includegraphics[width=0.32\textwidth]{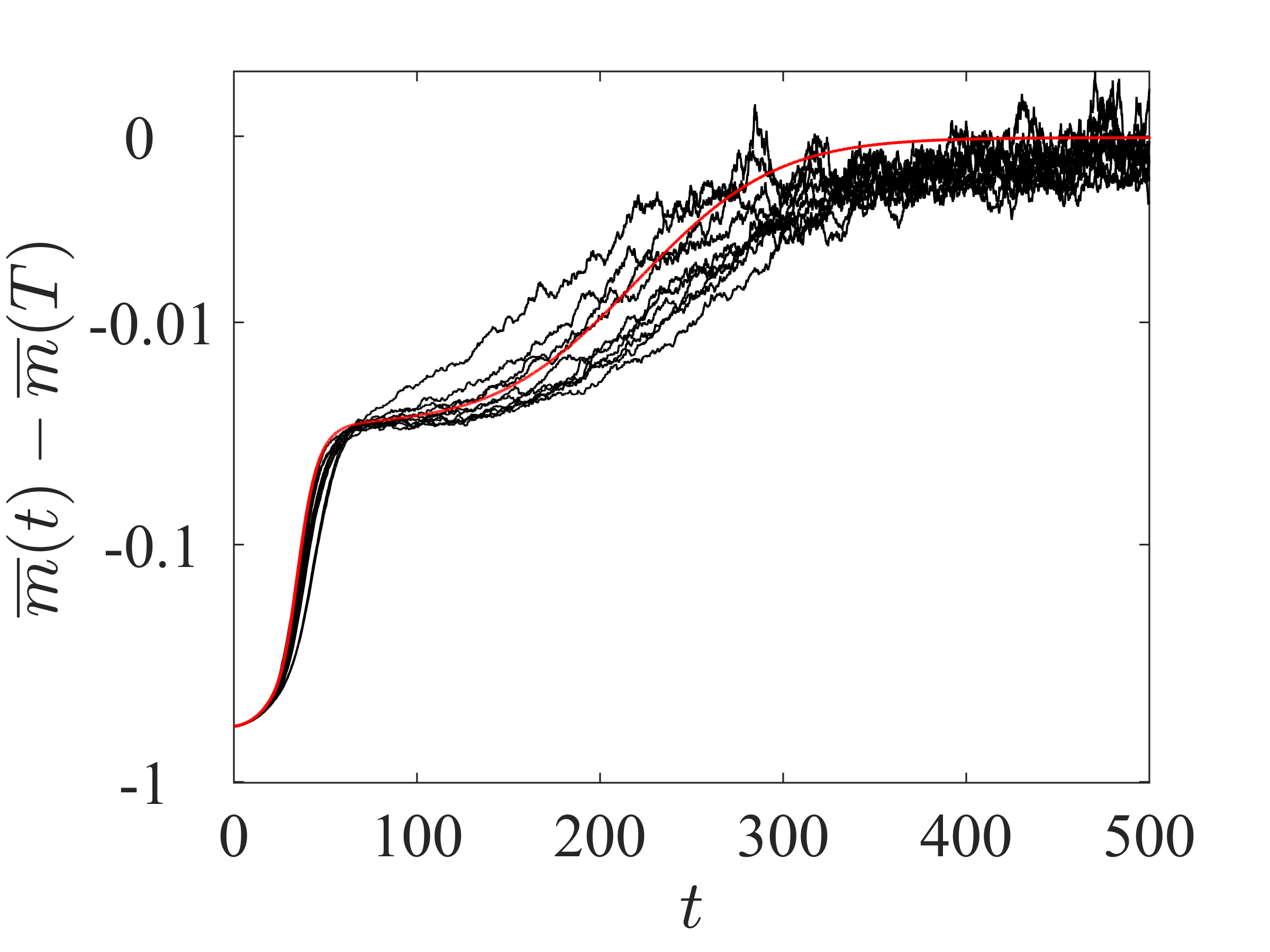}
		\includegraphics[width=0.32\textwidth]{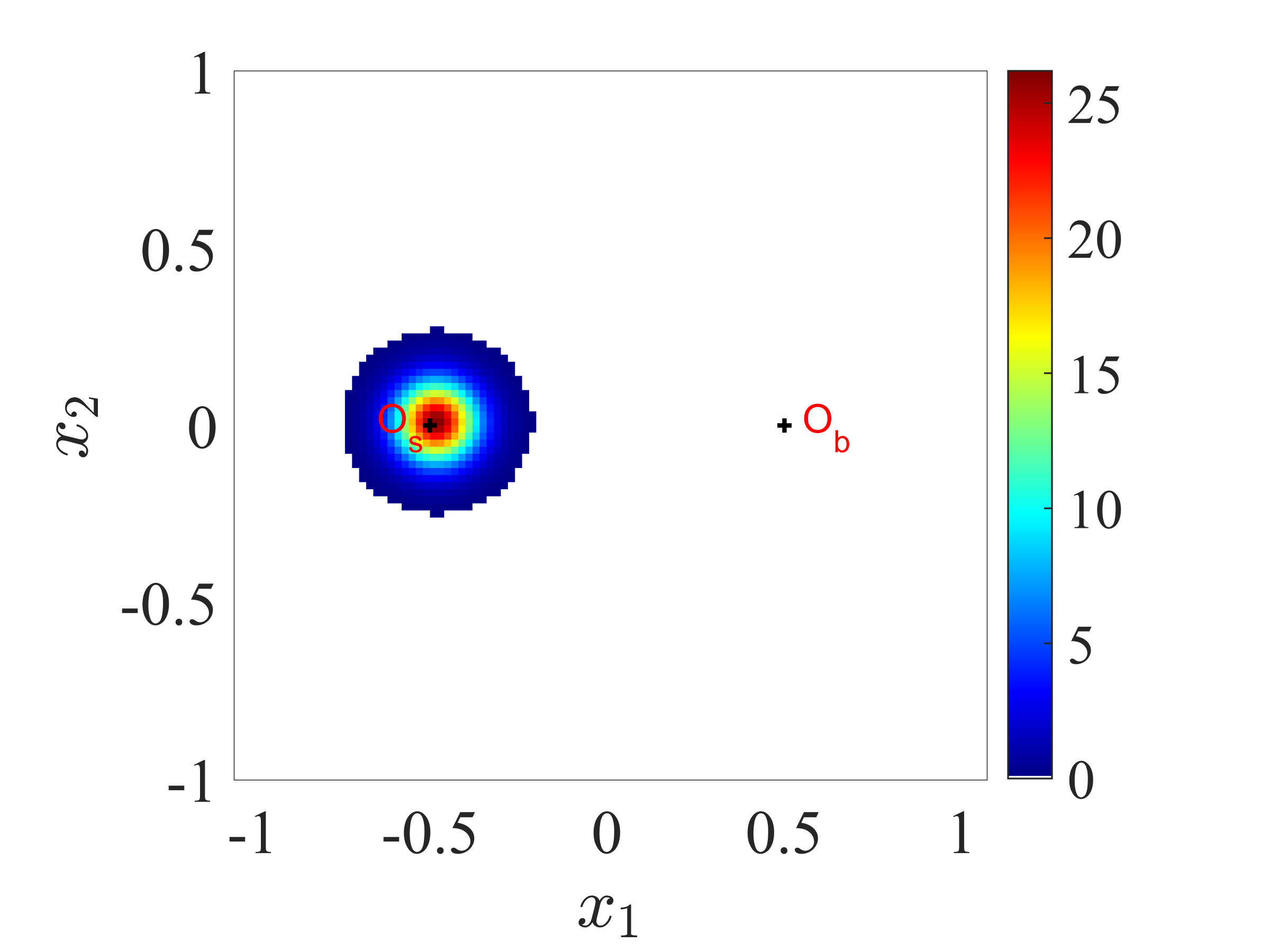}
		}
		\subfigure[Model  \eqref{eq:q}]{
		\includegraphics[width=0.32\textwidth]{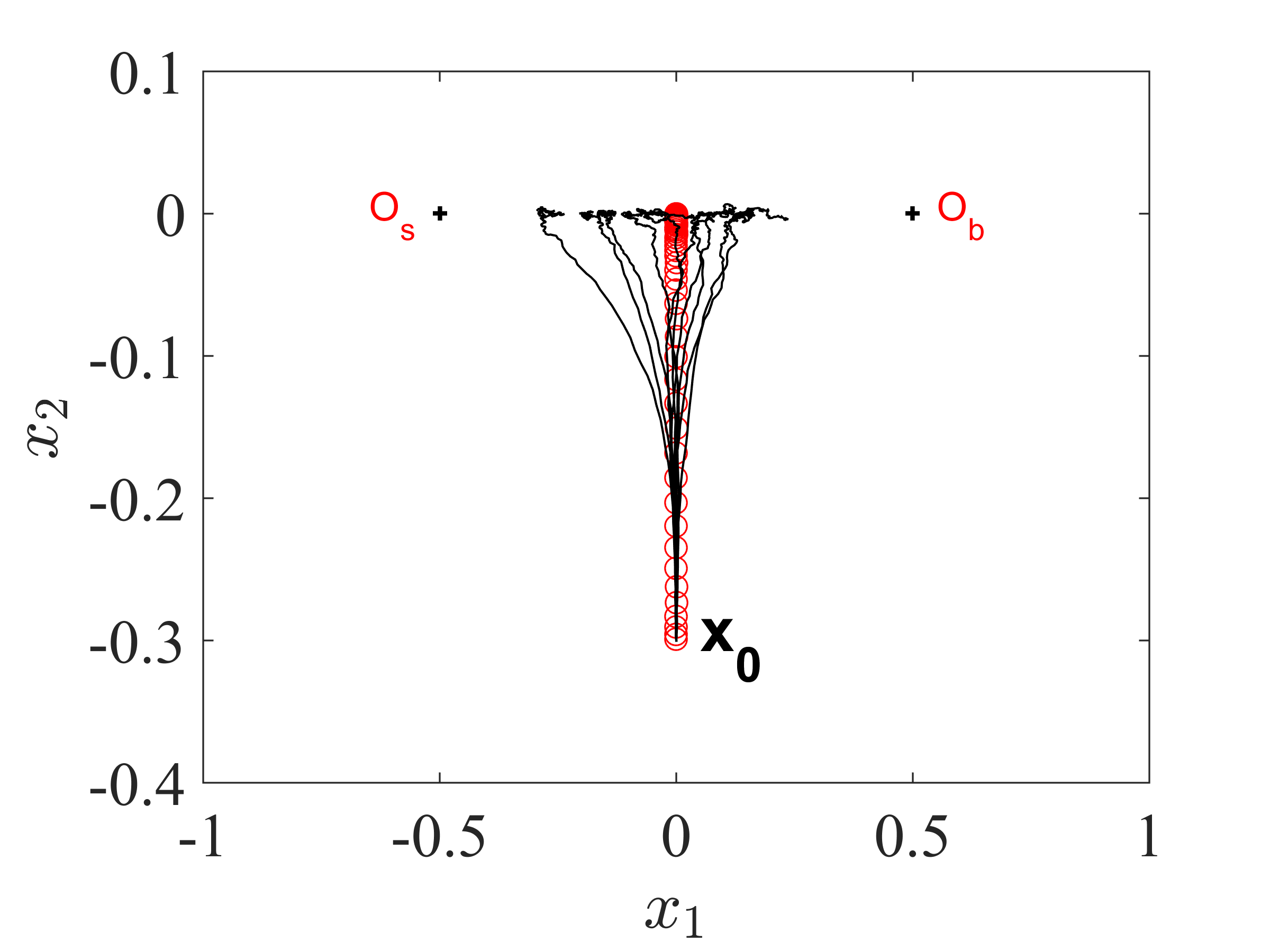}
		\includegraphics[width=0.32\textwidth]{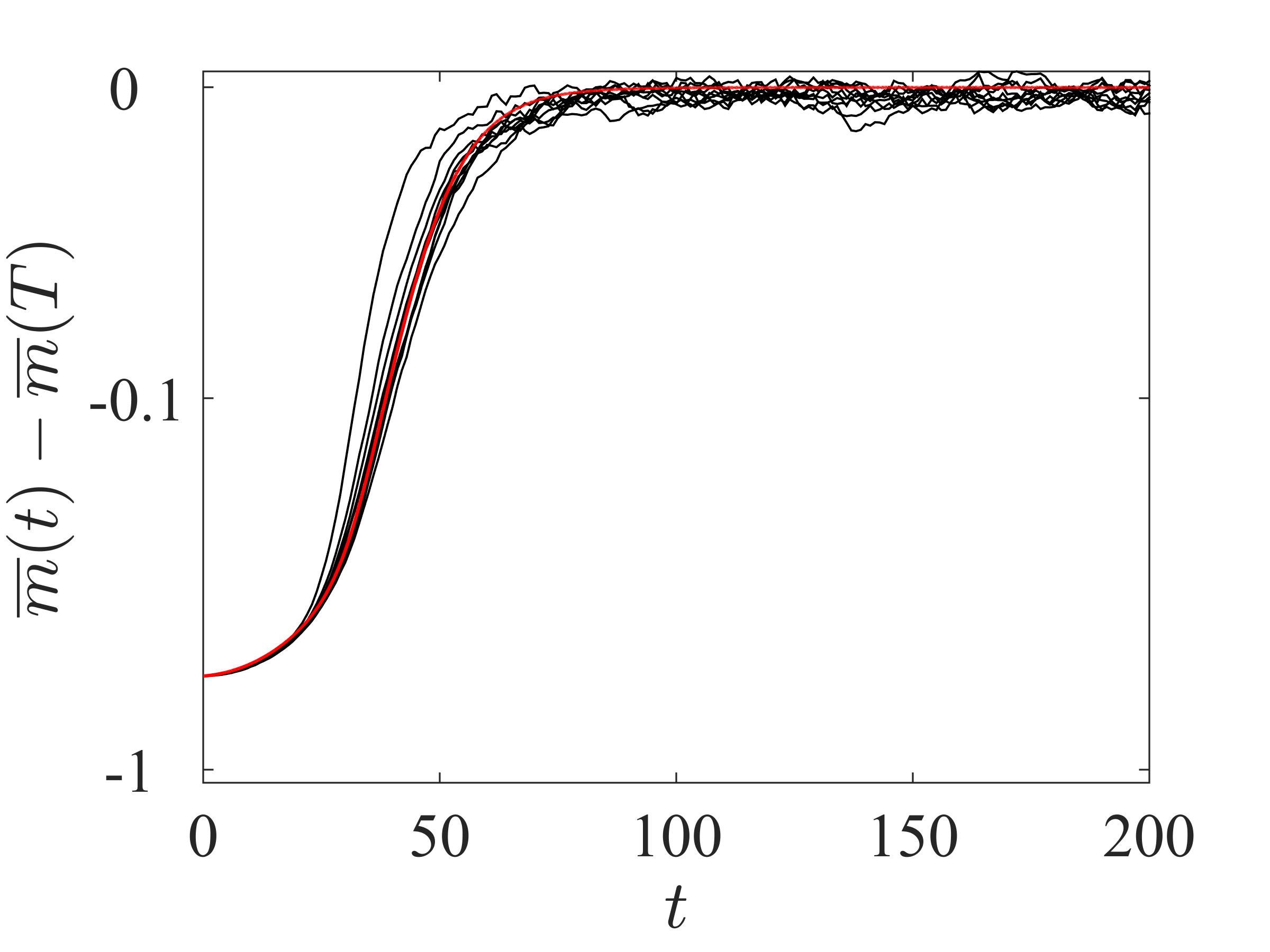}
		\includegraphics[width=0.32\textwidth]{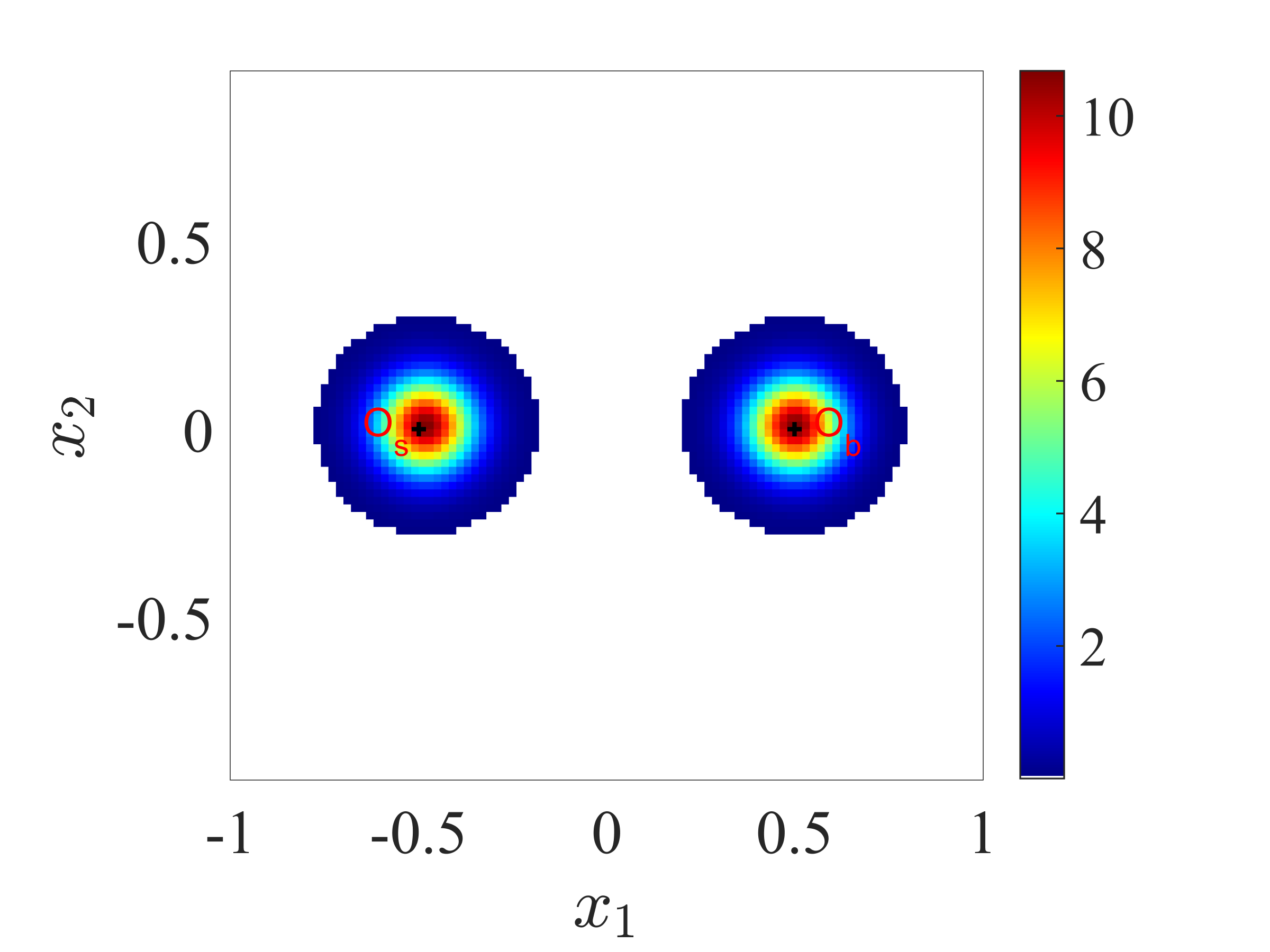}}
		\caption{{\bf Trajectory of adaptation and stationary distribution: model \eqref{eq:main_model} with birth-dependent mutation rate vs  standard model \eqref{eq:q}.}   The red circles in the left panels correspond to the position of the mean phenotype $\xb(t)$ with initial condition concentrated at   $\x_0=(0,-0.3)$,  at successive times $t=0,5,\ldots,T=500$ (upper panel) and $t=0,2,\ldots,T=200$ (lower panel). The \Rd red lines in the \Bk central panels describes the dynamics of the mean fitness (relative to its value at \Rd the final simulation time $T$),\Bk{} $\mb(t)-\mb(T)$, in a logarithmic scale. The black curves in these panels correspond to 10  replicate simulations of the individual-based models of \Cref{sec:causes}, with either overlapping generations (upper panels) or non-overlapping generations (lower panels). 
		The right panels correspond to the phenotype distribution at $t=T.$  We assumed here that the dimension is $n=2,$ $\Oc_b=(1/2,0)$, $\Oc_s=(-1/2,0)$ (i.e., $\beta=1/2$), $b(\x)=b(x_1,x_2)=b_0+\exp\left[-(x_1-\beta)^2/(2 \, \sigma_{x_1}^2)-x_2^2/(2 \, \sigma_{x_2}^2)\right]$, $s(\x)=s_0+\exp\left[-(x_1+\beta)^2/(2 \, \sigma_{x_1}^2)-x_2^2/(2 \, \sigma_{x_2}^2)\right]$,  $\sigma_{x_1}^2=\sigma_{x_2}^2=1/10$, $b_0=s_0=0.7$, $r=1+s_0$ and $D=2.4\cdot 10^{-4}$.  In the individual-based settings, we assumed a Gaussian mutation kernel with variance $\lambda=6\cdot 10^{-4}$ and a mutation rate $U=0.8$ so that $D=\lambda\, U/2$. }
		\label{fig:traj_bdep}
	\end{figure}

\paragraph{Initial bias towards the birth optimum} One of the qualitative properties observed in the simulations (Fig.~\ref{fig:traj_bdep}a) is an initial tendency of the trajectory of the mean phenotype $\xb(t)$ to go towards the birth optimum $\Oc_b$. We show here that this is a general feature, conditioned by the shape of selection along other dimensions. For simplicity, we denote by $\xxb(t)$ the mean value of the first trait, that is, the first coordinate of $\xb(t)$. We consider initial conditions $q_0$ that are symmetric about the hyperplane $\{x_1=0\}$, and that are localised around a phenotype $\x_0\in \{x_1=0\}$. By localised, we mean that $q_0$ vanishes outside some compact set that contains $\x_0$. We denote by $K_0$ the support of $q_0$, and define  $K_0^+:=K_0 \cap \{x_1>0\}$ the `right part of $K_0$'. We prove the following result (the proof is detailed in Appendix~\ref{proof initial}).
\begin{proposition}\label{prop:initial_bias}
	Let $q$ be the solution of \eqref{eq:main_model}, with an initial condition $q_0$ which satisfies the above assumptions. Then the following holds.
	\begin{itemize}
	    \item If $\Delta (x_1 m) \geq 0$ (and $\not \equiv 0$) on $K_0^+$, 
	    	    then the solution is initially biased towards the \underline{birth} optimum, that is
	    	\begin{equation}
	   \xxb'(t=0)=0\quad  \hbox{ and }\quad \xxb''(t=0)>0.
	\end{equation}
	\item If $\Delta (x_1 m) \leq 0$ (and $\not \equiv 0$) on $K_0^+$, then the solution is initially biased towards the \underline{survival} optimum, that is
	  	\begin{equation}
	   \xxb'(t=0)=0\quad \hbox{ and }\quad \xxb''(t=0)<0.
	\end{equation}
	\end{itemize}
\end{proposition}
A surprising feature of this proposition is the discussion around the sign of the quantity $\Delta (x_1 m)$. It shows that the local convexity (or concavity) of $m$ around the initial phenotype is important. It stems from the overall shape and symmetry of $m$. We first illustrate this in dimension $1$. In that case, the Laplace operator simply becomes
\[
\Delta (x m) = x m''(x) + 2 m'(x)=:g(x),
\]
By the symmetry assumption \eqref{eq:sym_b_s}, we know that $m'(0)=0$ and thus $g(0)=0$.  Therefore, in this one dimensional case, the discussion of \Cref{prop:initial_bias} about the sign of $\Delta(x m)$ is  linked to the
sign of $g'(0)=3m''(0)$, that is the local convexity of $m$ around $0$. Equivalently, it is also dictated by a discussion about the shape of $m$: if  $m$ presents a profile with two symmetric optima (camel shape, Fig.~\ref{fig:camel}a) or a single one located at $0$ (dromedary shape, Fig.~\ref{fig:camel}b), the outcome of the initial bias is different. If $m$ has a camel shape, then necessarily $m$ admits a local minimum around $0$. Therefore, as a consequence of \Cref{prop:initial_bias}, there is an initial bias towards the birth optimum.
On the other hand, if $m$ has a dromedary shape, the critical point $0$ is also a global maximum of $m$. From \Cref{prop:initial_bias}, it means that there is an initial bias towards survival.

This can be explained as follows. In the case where $m$ has two optima, the population is initially around a minimum of fitness. By symmetry of $m$ there is no fitness benefice of choosing either optimum. However,
individuals on the right have a higher mutation rate, which generates variance to fuel and speed-up adaptation, which explains the initial bias towards right. On the other hand, if $0$ is the unique optimum of the fitness function, the initial population is already at the optimum. Thus, generating more variance does not speed-up adaptation, but on the contrary generates more mutation load. \Rd Due to mutations,  the population cannot remain at the initial  optimum.  As the individuals on the left have a smaller mutation rate,  they tend to remain closer to the optimum, and therefore increase in proportion compared to the individuals on the right.  This explains the initial bias towards left.  As we will see with the spectral analysis in the next section, this bias towards left becomes a general feature at large times, independently of the conditions in \Cref{prop:initial_bias}. \Bk


In a multidimensional setting,  we can follow the same explanations, even if another phenomenon can arise.  The reason lies in the following formula:
\begin{equation}\label{eq:lap_exp}
    \Delta(x_1 m) = \p_{x_1 x_1}(x_1 m) + x_1 \sum_{j\geq 2} \p_{x_jx_j} m.
\end{equation}
Suppose that, as in Fig.~\ref{fig:camel}a, there is a local minimum around $\x_0$, in the first dimension. Then, the first term of \eqref{eq:lap_exp} is positive in a neighbourhood of $\x_0$ as soon as $x_1>0$, as we explained previously. In dimension $n\ge 2$, if the sum of the second derivatives with respect to the other directions is negative, the overall sign of $\Delta(x_1 m)$ may be changed. Such a situation can arise in dimension 2 if $\x_0$ is a saddle point. This phenomenon can be observed on Fig.~\ref{fig:diff scenar}. In both Fig.~\ref{fig:diff scenar}a and Fig.~\ref{fig:diff scenar}b, the fitness function $m$ is camel-like along the first dimension, as pictured in Fig.~\ref{fig:camel}a. However, as a consequence of the second dimension, we observe, or not, an initial bias towards the birth optimum. Similarly to the one-dimensional case, if the mutational load is too important, here on the second dimension, we do not observe this initial bias. This of course cannot be if $\x_0$ is a local minimum of $m$ in $\R^n$.

	\begin{figure}
		\center
		\subfigure[Shape 1: $\Delta (x_1 m(x_1,x_2))>0$ on the axis $\{x_1=0\}$]{\includegraphics[width=0.33\textwidth]{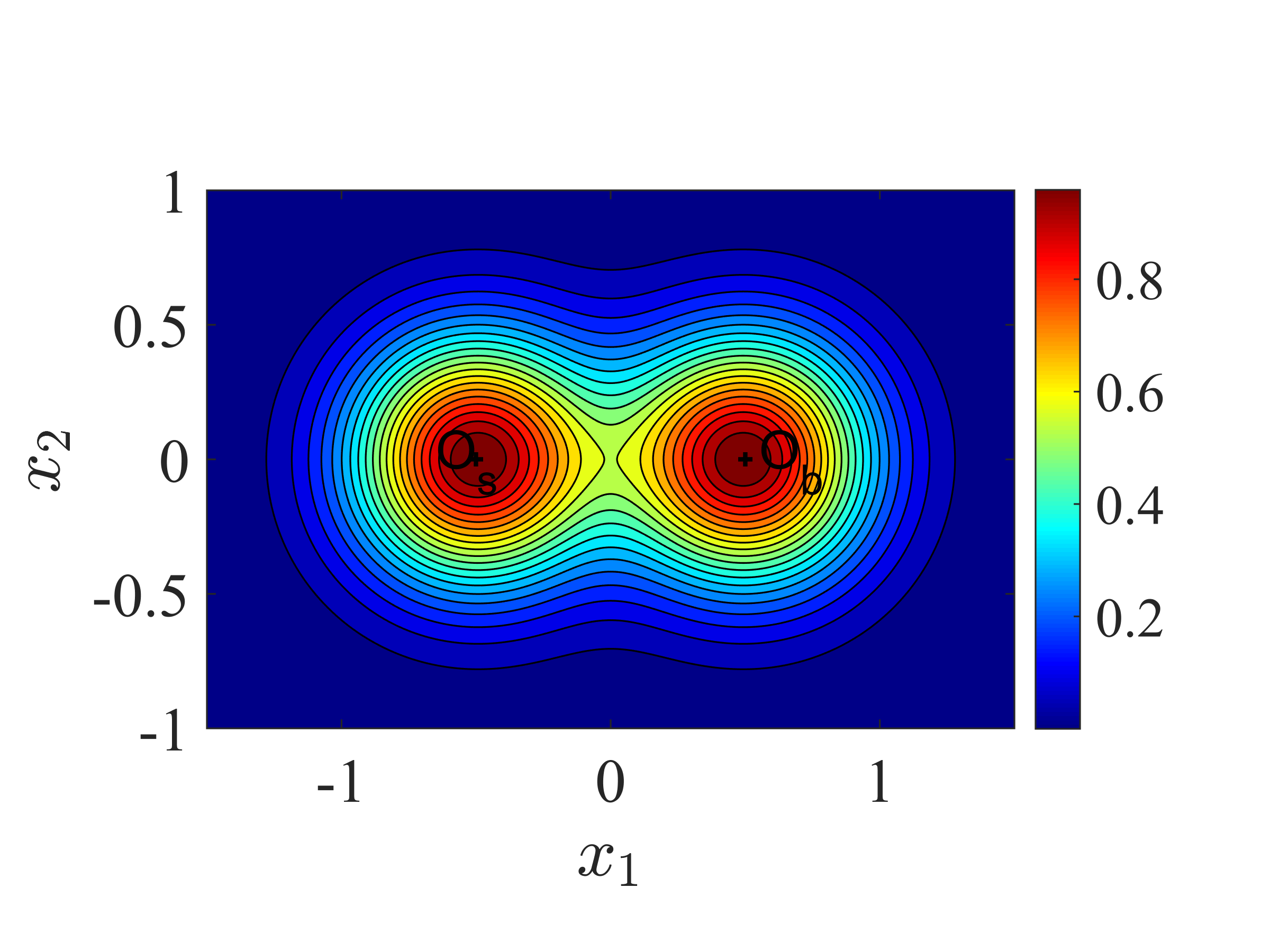}\includegraphics[width=0.33\textwidth]{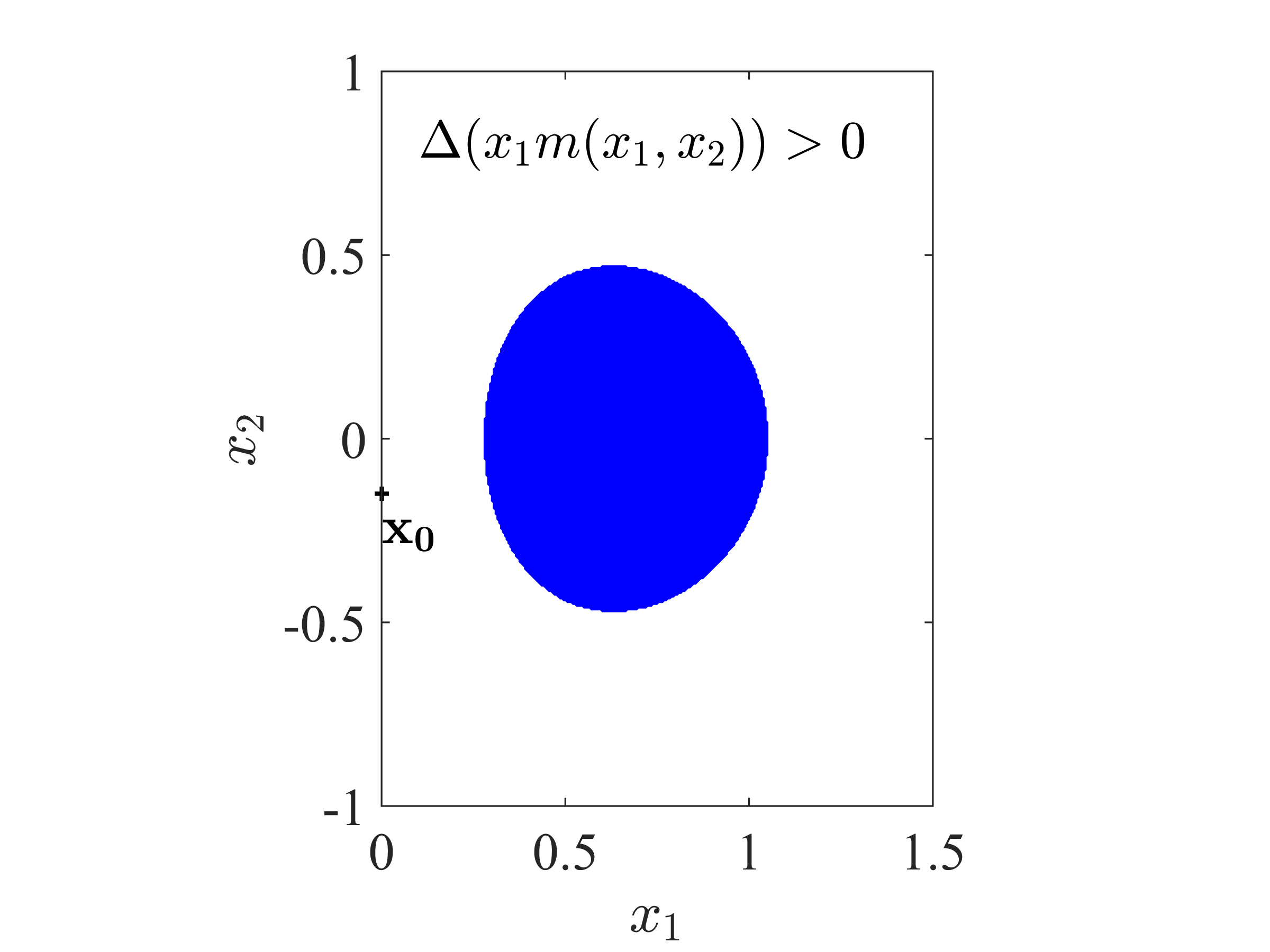} \includegraphics[width=0.33\textwidth]{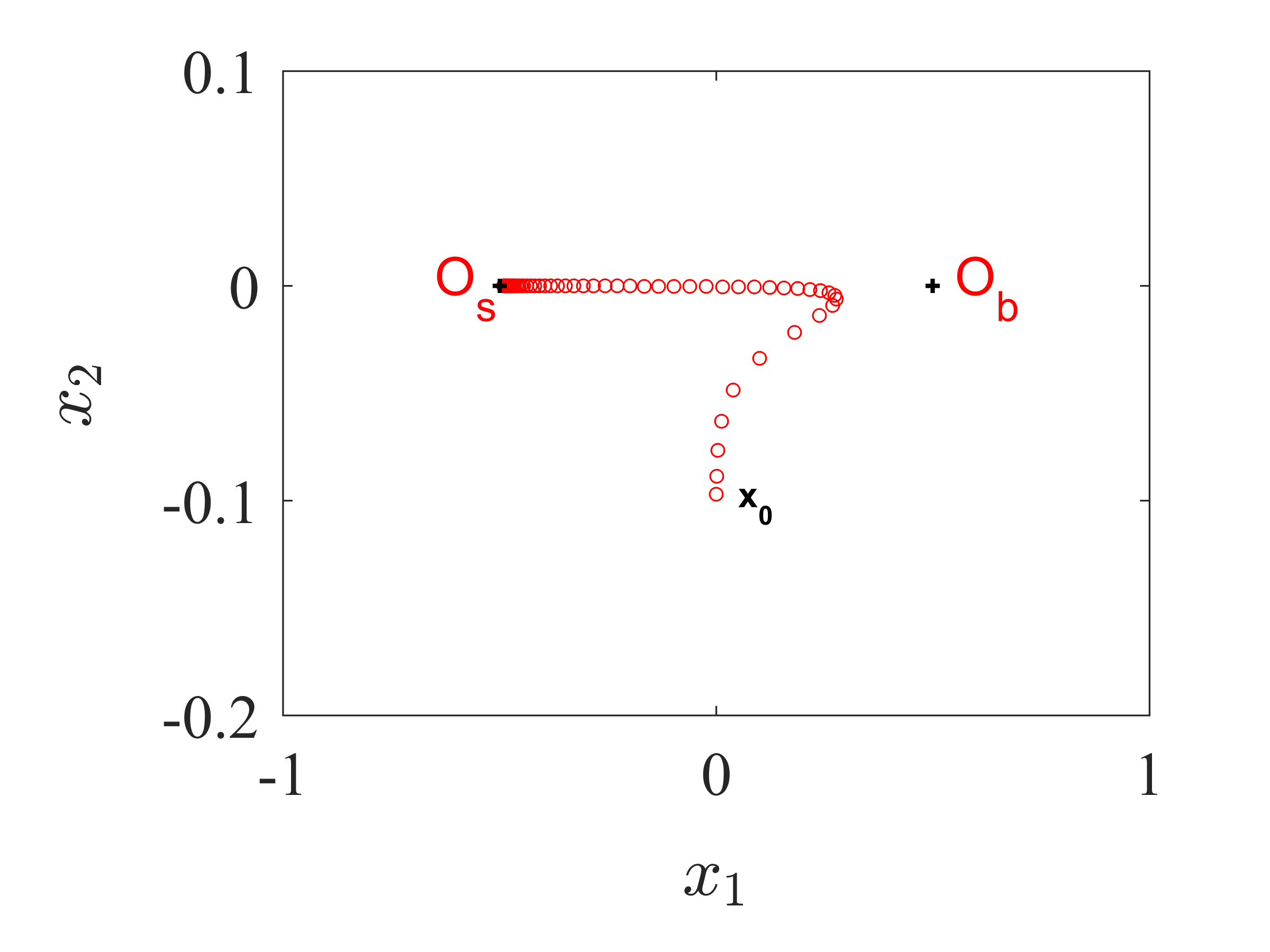} }
fig		\subfigure[Shape 2: $\Delta (x_1 m(x_1,x_2))$ changes sign on the axis $\{x_1=0\}$ ]{\includegraphics[width=0.33\textwidth]{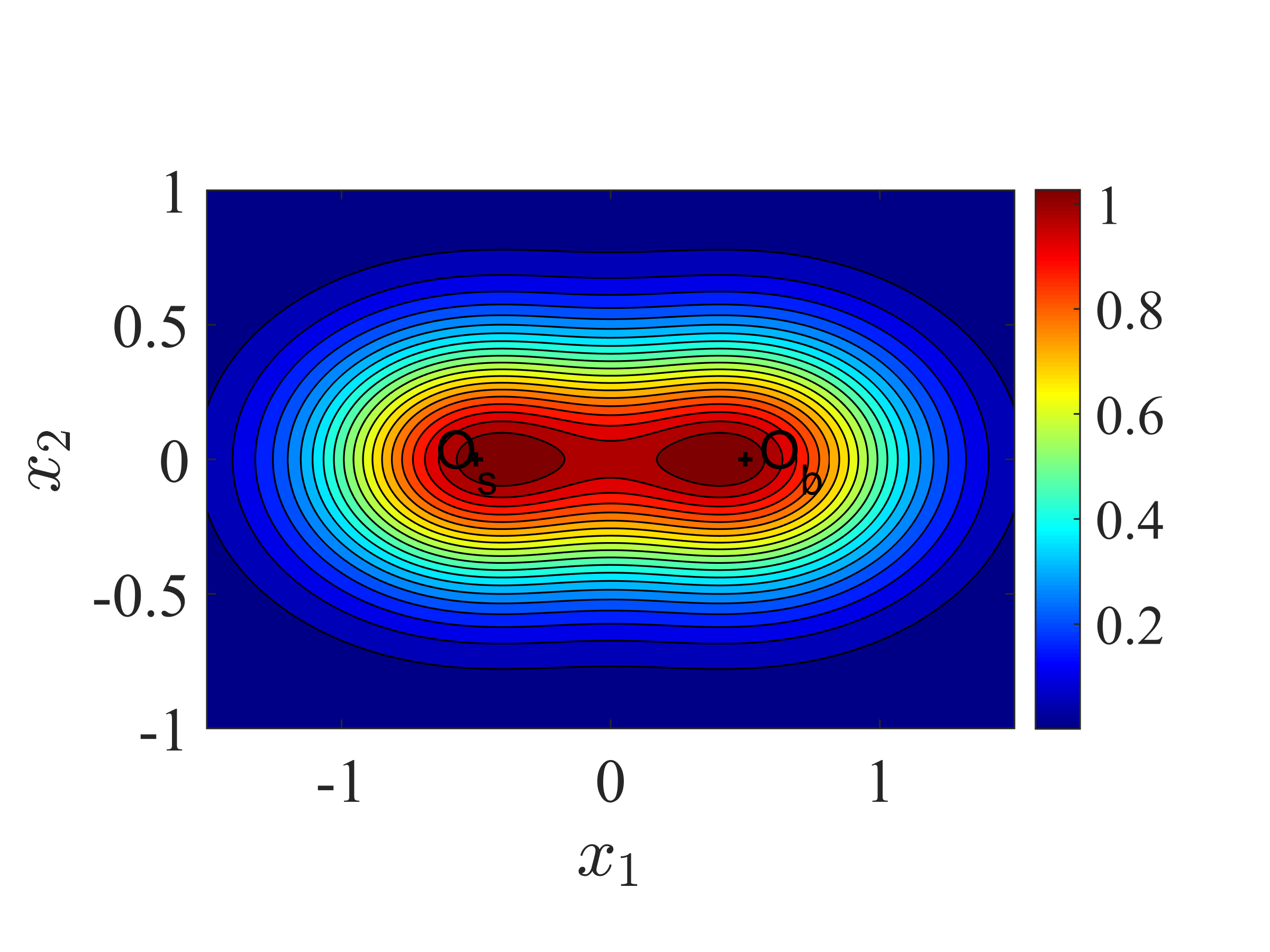}\includegraphics[width=0.33\textwidth]{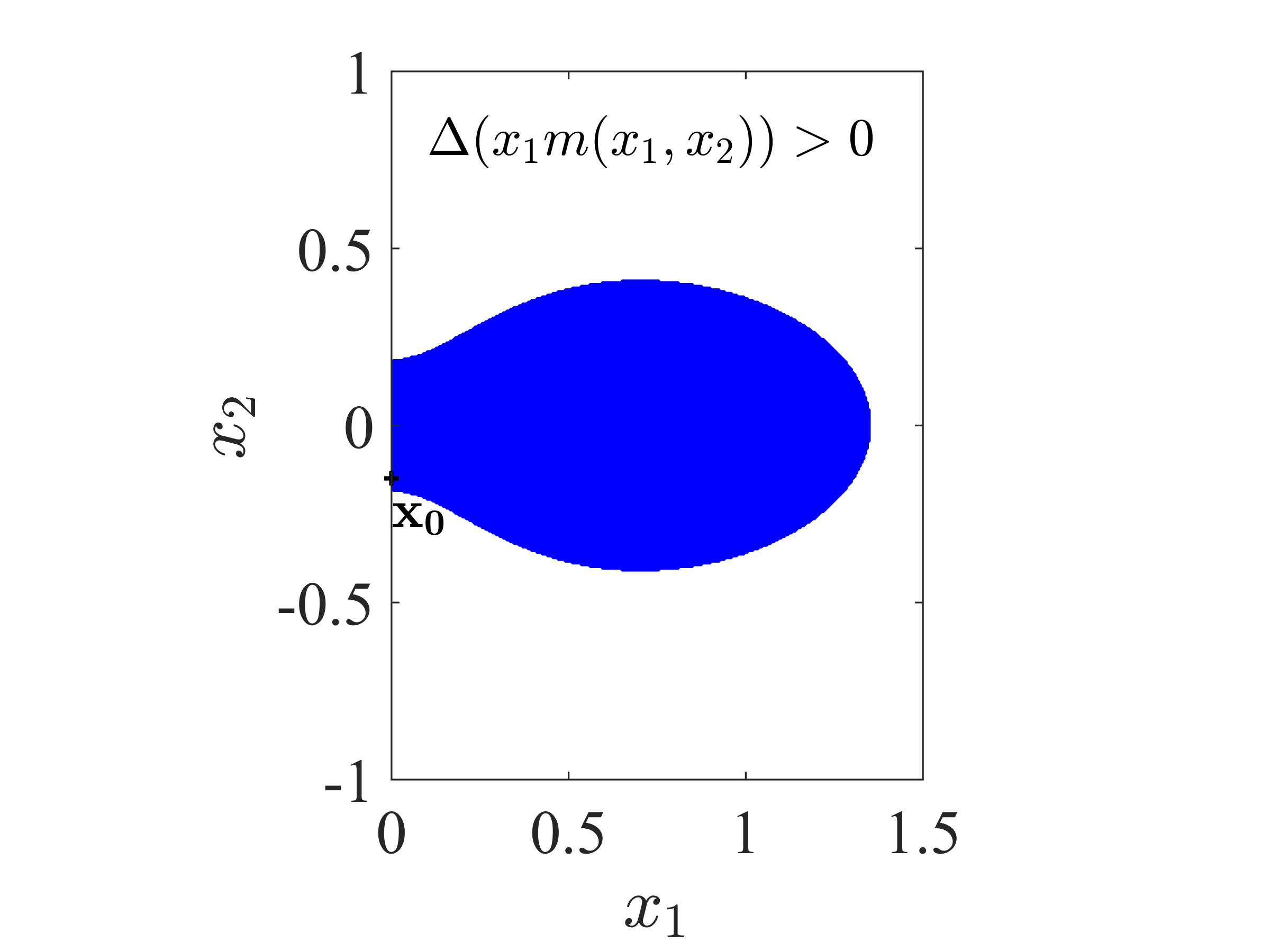} \includegraphics[width=0.33\textwidth]{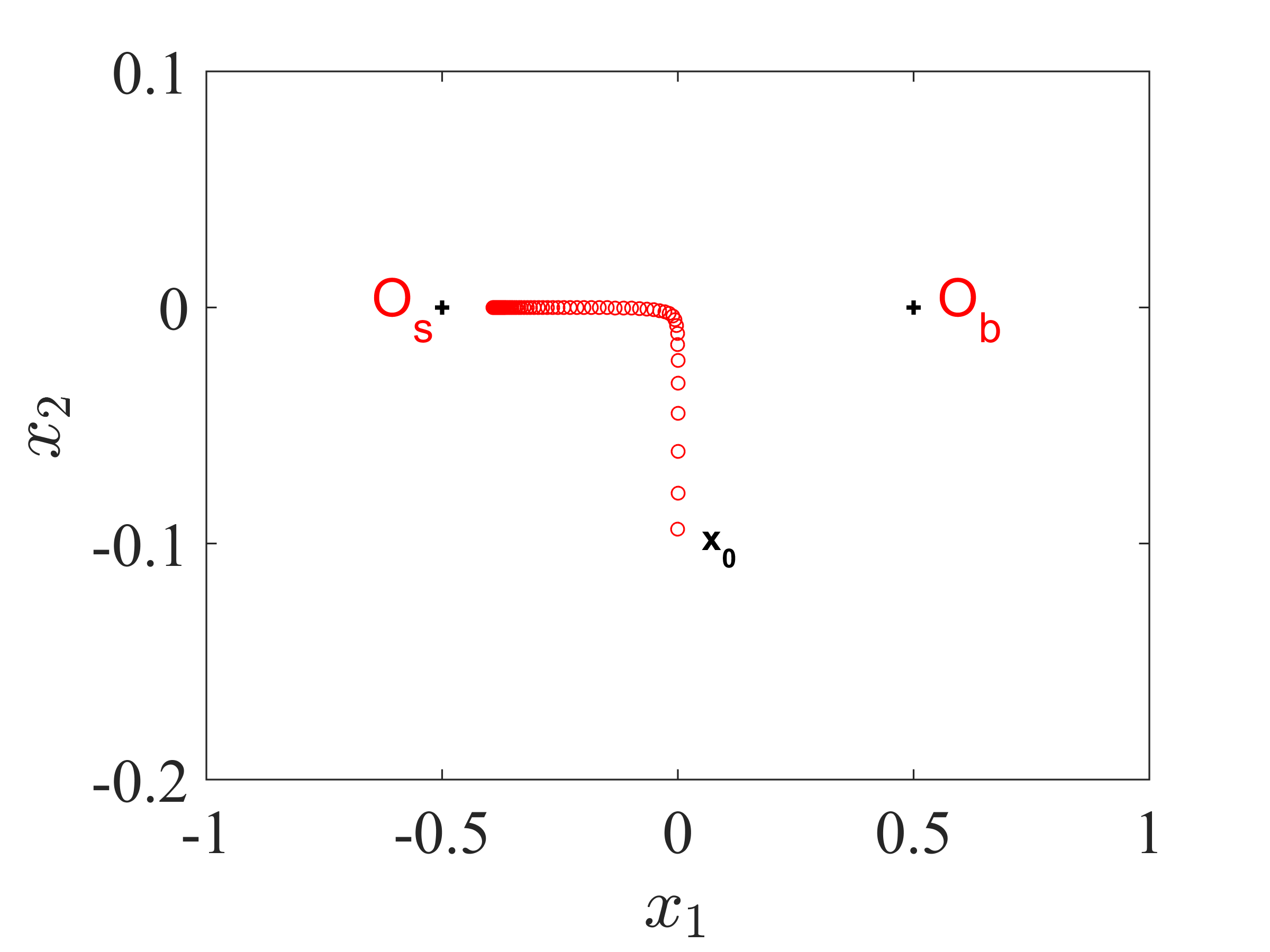}}
		\caption{{\bf Trajectory of adaptation with different shapes of the fitness function.} The left panels depict the fitness function. The central panels describe the sign of $\Delta (x_1 m(x_1,x_2))$ in the region $\{x_1>0\}$: this quantity is negative in the blue region and positive otherwise. The right panels depict the corresponding trajectories of the mean phenotype $\xb(t)$ obtained with the model \eqref{eq:main_model}, at successive times $t=0,1,\ldots,100$. In both cases, the initial condition is concentrated at $\x_0=(0,-0.1)$, leading to a positive sign of $\Delta (x_1 m(x_1,x_2))$ at $(x_1,x_2)=\x_0$ in the upper panels and a negative sign in the lower panels. The parameter values are the same as in Fig.~\ref{fig:traj_bdep}, except for the fitness function of the lower panel, where $\sigma_{x_1}^2=1/18$ and $\sigma_{x_2}^2=1/10$.}
		\label{fig:diff scenar}
	\end{figure}

%


\paragraph{Large time behaviour} We now analyse whether the convergence towards the survival optimum at large times observed in Fig.~\ref{fig:traj_bdep}a is a generic behaviour. In that respect, we focus on the stationary distribution $q_\infty$ associated with the model~\eqref{eq:main_model}. It satisfies equation
\begin{equation} \label{eq:statio1}
	D\, \Delta (b \, q_\infty) (\x)+m(\x)\, q_\infty(\x) = \mb_\infty \, q_\infty(\x),  \quad \x\in\Omega,
\end{equation}
for some $\mb_\infty \in \R$. Setting
\[
v(\x):=b(\x) \, q_\infty(\x),
\]
this reduces to a more standard eigenvalue problem, namely
	\begin{equation} \label{eq:statio2}
	D\, \Delta v (\x)+ \frac{m(\x)}{b(\x)} v(\x)= \mb_\infty \, \frac{1}{b(\x)} \,  v(\x),  \quad \x\in\Omega,
	\end{equation}
supplemented with Neumann boundaries conditions:
\begin{equation}\label{neumann}
\nabla v(\x)\cdot \nub(\x)=0, \quad x\in \partial \Omega.
\end{equation}
As the factor $1/b(\x)$ multiplying $\mb_\infty$ is strictly positive, we can indeed apply the standard spectral theory of  \cite{CouHil53} \citep[see also][]{CanCos03}. Precisely, there is a unique couple $(v(\x), \mb_\infty)$ satisfying \eqref{eq:statio2}-\eqref{neumann} (with the normalisation condition $\int_\Omega v(\x)d\x=1$) such that $v(\x)>0$ in $\Omega$. The `principal eigenvalue' $\mb_\infty$ is provided by the variational formula
	\begin{equation} \label{eq:variational2}
	\mb_\infty=\max_{\psi \in W^{1,2}(\Omega)} Q[\psi],
	\end{equation}
where $W^{1,2}(\Omega)$ is the standard Sobolev space and
\[
Q[\psi]:=\frac{-D\, \int_{\Omega} \|\nabla ( \psi \, \sqrt{b})\|^2 (\x) d\x +\int_{\Omega}  m(\x) \psi^2 (\x) d\x }{\int_{\Omega}  \psi^2 (\x) d\x }.
\]
An immediate consequence of formula \eqref{eq:variational2} is that $\mb_\infty$ is a decreasing function of the mutational parameter $D$. This means that, as expected, the mutation load increases when the mutational parameter is increased. 

We expect the stationary state to  `lean mainly on the left', meaning that the survival optimum is selected at large times, but deriving rigorously the precise shape of $q_\infty$ seems highly involved. Still,  formula~\eqref{eq:variational2} gives us some intuition. First,
multiplying~\eqref{eq:statio2} by $v$ and integrating, we observe that $Q[v/\sqrt{b}]=Q[\sqrt{b} \, q_\infty]=\mb_\infty$. Thus, formula~\eqref{eq:variational2} shows that the shape of $q_\infty$ should be such that $\psi=\sqrt{b} \, q_\infty$ maximizes the Rayleigh quotient $Q$. 

We thus consider each term of $Q$ separately. From Hardy-Littlewood-P\'olya rearrangement inequality, the term $\int_{\Omega}  m(\x) \psi^2 (\x) d\x$ is larger when $\psi$ is arranged like $m$, i.e., $\psi$ takes its largest values where $m$ is large and its smallest values where $m$ is small. Thus, this term tends to promote shapes of $\psi$ which look like $m$. The  other term $-D\, \int_{\Omega} \|\nabla ( \psi \, \sqrt{b})\|^2 (\x) d\x$ tends to promote functions $\psi$ which are proportional to $1/\sqrt{b}$. Finally, the stationary distribution $q_\infty$ should therefore realise a compromise between $1/b$ and $m/\sqrt{b}$. As both functions take their larger values when $b$ is small, we expect $q_\infty$ to be larger close to the survival optimum $\Oc_s$.

More rigorously, define $\tq(\x)=q(-x_1,x_2,\ldots,x_n)=q(\iota(\x))$. As $\sqrt{b} \, q_\infty$ realises a maximum of $Q$, we have \[Q[\sqrt{b} \, q_\infty]\ge Q[\sqrt{s} \, \tq_\infty].\] Recalling $s(\x)=b(\iota(\x))$ and using the symmetry of $m$, this implies that
\begin{equation}\label{eq:moral_ineq}
\int_{\Omega} \|\nabla (  \tq_\infty \sqrt{b\, s})\|^2 =  \int_{\Omega} \|\nabla (  q_\infty \sqrt{b\, s})\|^2 \ge \int_{\Omega} \|\nabla ( q_\infty \, b)\|^2.    
\end{equation}
Now, we illustrate that \textit{moralement} this gradient inequality means that the stationary distribution tends to be closer to $\Oc_s$ than to $\Oc_b$. In dimension $n=1$, assume that $b(x)=\exp(-(x-\beta)^2)$ and $s(x)=\exp(-(x+\beta)^2)$. Assume that the domain is large enough so that the integrals over $\Omega$ can be accurately approached by integrals over $(-\infty,+\infty)$. Among all functions of the form $h_\gamma(x)=\exp(-(x-\gamma)^2)$,	a straightforward computation reveals that
\[
\lp\int_{-\infty}^{+\infty}   [\p_x (h_\gamma \, \sqrt{b\, s})]^2 (x) \, dx \ge \int_{-\infty}^{+\infty}   [\p_x (h_\gamma \, b)]^2 (x) \, dx\rp \  \Leftrightarrow \ \gamma \leq - \beta/2,
\]
which means that the inequality \eqref{eq:moral_ineq} is satisfied by functions $h$ whose maximum is reached at a value $x=\gamma$ closer to $\Oc_s=-\beta$ than to $\Oc_b=\beta$.

\paragraph{Large mutation effects} This advantage of adaptation towards the survival optimum becomes more obvious when the mutation effects are large. We observed above that $\mb^D_\infty$ (seen here as a function of $D$) is decreasing. Moreover, from  \eqref{eq:variational2}, we have, for all $D>0$,
\[
\mb^D_\infty \geq Q[1/\sqrt b]=\lp\int_\Omega mb^{-1}  \rp \lp \int_\Omega b^{-1}\rp^{-1}.
\]
Thus $\mb^D_\infty$ admits a limit $\mb^\infty_\infty$ as $D\to \infty$. Moreover, the corresponding stationary states satisfy $\Delta (b \, q^D_\infty) (\x)+q^D_\infty(\x)\, (m(\x)-\mb^D_\infty)/D = 0$. Standard elliptic estimates and Sobolev injections imply that, up to the extraction of some subsequence $D_k\to \infty$, the functions $q^{D_k}_\infty$ converge, as $k\to\infty$, in $C^{2}(\Omega)$ to a nonnegative solution (with mass $1$) of $\Delta (b \, q^\infty_\infty) (\x) = 0$ \Rd{} with $\nabla (b \, q^\infty_\infty)(\x)\cdot \nub(\x)=0$ on $\p \Omega$. \Bk{} As such a solution is unique and given by:\[q^\infty_\infty(\x)=C/b(\x) \hbox{ with }C=\int_\Omega b^{-1},\]the whole sequence $q^D_\infty$ converges to  $C/b(\x)$ as $D\to \infty$. Thus, in order to reduce the mutation load, the phenotype distribution tends to get inversely proportional to $b$ in the large mutation regime.

\paragraph{An analytically tractable example}  Consider the following form for the birth rate, in dimension $n=1$:
	\begin{equation} \label{eq:b_constant_morceaux}
	b(x)=\baco{l}
	2 \qquad \hbox{ for }   x\in  (0,a),\\
	1 \qquad \hbox{ for }   x\in (-a, 0),\\
	-M \;\,\,  \hbox{ for }  x\not\in (-a, a).
	\eaco
	\end{equation}
With the assumptions~\eqref{eq:def_m=r+b+s} and~\eqref{eq:b_constant_morceaux}, we get:\[m(x)=3-r, \hbox{ for }x\in (-a,a),\]and $m(x)=-2\, M-r$ outside $(-a,a)$. Then, we consider the corresponding 1D eigenvalue problem \eqref{eq:statio1} in an interval $\Omega$ containing $(-a,a)$. Assuming that the phenotypes are extremely deleterious outside $(-a,a)$ (i.e., $M\gg 1$), we make the approximation $q_\infty(\pm a)=0$. \Rd{} Thus, the shapes of the birth and survival functions exactly compensate each other, so that the fitness function is constant in the interval $(-a,a)$. \Bk
In this case, we prove (see Appendix~\ref{proof explicit}) that
\[\ds\frac{\int_{-a}^0 q_{\infty}(x)\, dx}{\int_{0}^a q_{\infty}(x)\, dx}   > \frac{1}{2 \,\sqrt{2-\sqrt{2}}-2+\sqrt{2}}>1.\]In other word, the stationary distribution has a larger mass to the left of $0$ (where $s$ is larger) than to the right (where $b$ is larger). \Rd{} Thus, even with a constant fitness function and therefore in the absence of mutation load, adaptation tends to promote the high survival strategy, probably because of higher fluxes from the regions of the phenotype space with a high birth rate (and therefore a high mutation rate) to the regions with a high survival rate. More generally, it is easily observed that if $m(\x)$ is constant in \eqref{eq:statio1}, then $q_\infty(\x)$ is proportional to $1/b(\x)$ and therefore takes larger values when $b$ is small, showing that the high survival strategy is promoted at equilibrium. In Appendix~\ref{app:num_flat}, we illustrate this property and depict the dynamics of adaptation for a particular example of function $b$. \Bk

\section{Discussion}\label{sec conc}

We found that a positive dependence between the birth rate and the mutation rate emerges naturally at the population scale, from elementary assumptions at the individual scale.
Based on a large population limit of a stochastic individual-based model  in a small mutation variance regime  we derived a reaction-diffusion framework \eqref{eq:main_model} that describes the evolutionary trajectories and steady states in the presence of this dependence. We compared this approach with stochastic replicate simulations of finite size populations which showed a good agreement with the behaviour of the reaction-diffusion model. These simulations, and our analytical results on \eqref{eq:main_model} demonstrate that taking this dependence into account,  or conversely omitting it as in the standard model  \eqref{eq:q}, has far reaching consequences on the description of the evolutionary dynamics. In light of our results, we discuss below the causes and consequences of the positive dependence between the birth rate and the mutation rate.

\paragraph{Birth-dependent mutation rate: causes} Even though the probability of mutation per birth event $\Ui$ does not depend on the phenotype of the parent, and therefore on its fitness nor its birth rate, a higher birth rate implies more mutations per unit of time at the population scale. This holds true when mutations mainly occur during reproduction, which is the case for bacteria and viruses \citep{Van98,TruTre09}.
 The mathematical derivation of the standard model \eqref{eq:q}, that does not account for this dependence, generally relies on a weak selection assumption, which {\it de facto} implies a very mild variation of the birth rate with the phenotype.
More precisely, the mutation variance and the difference between birth rates and death rates should both be small and of comparable magnitude, uniformly over the phenotype space explored by the population (see Assumptions SE and WS in Section~\ref{sec:causes}).
This is usually achieved by assuming that the leading order in the birth rate does not depend on the phenotype. 
In such cases, the mutation rate can safely be assumed to be phenotype-independent at the population scale, even though it is positively correlated with the birth rate, as already observed in \citep{Hof85,BaaGab00}. When there is a single optimum, this weak selection regime is often relevant. In particular, a scaling of the phenotype space shows that taking small mutation effects is equivalent to having a weak selection. 
Thus, in a regime with small mutation variance, which is required for the diffusion approximation,  and with a single fitness optimum,  the models \eqref{eq:q} and \eqref{eq:main_model} should lead to very similar results. However, in a much more complex phenotype to fitness landscape with several optima, this approximation does not hold.
 In particular, if the birth rates at each optimum are very different from one another,  even with a small mutation variance, the mutation term $\Delta(b \, q)$ will be very different from one optimum to another. In such situations, our approach reveals that the model \eqref{eq:main_model} will be more relevant,  and lead to more accurate predictions of the behaviour of the individual-based model. 

An exception corresponds to organisms with non-overlapping generations: the simulations in Fig.~\ref{fig:traj_bdep}b indicate that even with a fitness function that strongly depends on the phenotype, the trajectories of adaptation are adequately described by the model \eqref{eq:q}. Species with non-overlapping generations include annual plants (but some overlap may exist due to seedbanks), many insect species \citep[e.g. processionary moths,][again some overlap may exist due to prolonged diapause]{RoqA15} and fish species \citep[such as some killifishes with annual life cycles,][]{TurWri15}.

\paragraph{Birth-dependent mutation rate: consequences} When the model \eqref{eq:main_model} is coupled with a phenotype to fitness landscape with two optima, one for birth, the other one for survival, a new trade-off arises in the population. Compared to the standard approach \eqref{eq:q}, the symmetry between birth and survival is broken. Thus, in a perfectly symmetric situation (symmetric initial condition and fitness function), our analytical results and numerical simulations show new nontrivial strategies for the trajectories of mean phenotype and for the stationary phenotype distribution. These new strategies are in sharp contrast with those displayed for the standard model \cref{eq:q} for which  the two optima are perfectly equivalent. With the model \eqref{eq:main_model}, we obtained trajectories of adaptation where the mean phenotype of the population is initially attracted by the birth optimum, but eventually converges to the survival optimum, following a hook-shaped curve (see Fig.~\ref{fig:traj_bdep}). It is well-known that increasing the  mutation rate has antagonistic effects  on adaptation in the FGM (and other models with both deleterious and beneficial mutations) as it generates fitness variance to fuel and speed-up adaptation \citep{LavMar20} but lowers the mean fitness by creating a larger mutation load \citep[e.g.][]{AncLam19}. Here, these two effects shape the trajectories of adaptation. When far for equilibrium, the phenotypes with a higher mutation rate tend to be advantaged, leading to trajectories of mean phenotype that are initially biased toward the birth optimum. Then, in a second time, adaptation promotes the survival optimum, as it is associated with a lower mutation load. 

In our study, the birth and survival functions have the same height and width, so that the resulting fitness landscape $m(\x)$ is double peaked and symmetric. We chose this particular landscape in order to avoid trivial advantageous effect for one of the two strategies, and to check if an asymmetrical behaviour can emerge from a symmetric fitness landscape. Of course, if the birth optimum corresponds to a much higher fitness than the survival optimum, we expect that at large times the mean phenotype will converge to the birth optimum. However, our approach shows the tendency of the trajectory to be attracted by the survival optimum, which clearly shows up in the symmetric case considered here, and remains true in intermediate situations, as observed in Fig.~\ref{fig:gamma}. More precisely, in an asymmetric double peaked fitness landscape where the two peaks have different height, we observed that having a high survival rate remains more advantageous at equilibrium than having a high birth rate as long as the difference between the fitness peaks remains lower than the difference between the mutation loads generated at each optimum (see Appendix~\ref{app:asym}).

Another feature of the model \eqref{eq:main_model} is that the transient trajectory of mean fitness displays plateaus, as observed for instance in Fig.~\ref{fig:traj_bdep}. This phenomenon of several epochs in adaptation is well documented thanks to the longest ever evolution experiment, the `Long Term Evolution Experiment' (LTEE). Experimenting on \textit{Escherichia coli} bacteria,  \cite{WisRib13} found out that even after more than $70,000$ generations, fitness had not reached its maximum, apparently challenging the very existence of such a maximum, the essence of Fisher's Geometrical Model. It was then argued that the data could be explained by a two epoch model \citep{GooDes15}, with or without saturation. A similar pattern was observed for a RNA virus \citep{NovDua95}. Recently, \cite{HamLav20} showed that the FGM with a single optimum but anisotropic mutation effects also leads to plateaus, and they obtained a good fit with the LTEE data. Our study shows that, when coupled with a phenotype to fitness landscape with two optima, the model \eqref{eq:main_model} is also a possible candidate to explain these trajectories of adaptation.

\paragraph{The model \eqref{eq:main_model} in the mathematical literature} Some authors have already considered operator which are closely related to the mutation operator in \eqref{eq:main_model}.  For instance \cite{LorMir11} considered non-homogeneous operators of the form $\mathcal B(q)(t,\x) = \text{ div} (b(\x) \nabla q(t,\x))$ within the framework of constrained Hamilton-Jacobi equations.	However this operator does not emerge as the limit of a microscopic diffusion process or as an approximation of an integral mutation operator. It is more adapted to the study of heat conduction as it notably tends to homogenise the solution compared to the Fokker-Planck operator $ \Delta( b(\x) q(t,\x) )$, see Figure II.7 in \cite{Roq13}. Finally, the flexible framework of \cite{burger2000mathematical} allows for heterogeneous mutation rate. Due to the complicated nature of the operator involved (compact or power compact kernel operator), the theoretical framework is in turn very intricate.
Quantitative results are in consequence either relatively few, and typically consist in existence and uniqueness of solutions, upper or lower bounds on the asymptotic mean fitness \citep{Bur98,burger2000mathematical},  or concern simpler models (with a discretisation of the time or of the phenotypic space), see \citep{HerRed02,Red04,Hof85}.

\paragraph{Sexual reproduction}  How to take into account a phenotype dependent birth rate with a sexual mode of reproduction is an open question to the best of our knowledge. A classical operator to model sexual genetic inheritance in the background adopted in this article is the \textit{infinitesimal operator}, introduced by  \cite{Fisher1918}, see \cite{slatkin1970selection,cotto-ronce} or the review of \cite{turelli_comm}. It describes a trait deviation of the offspring around the mean of the phenotype of the parents, drawn from a Gaussian distribution. Mathematically, few studies have tackled the operator, with the notable only exceptions of the derivation from a microscopic point of view of \cite{weberinfinitesimal}, the small variance and stability analysis of \cite{spectralsex,patout2020cauchy}, and  in \cite{raoulmirrahimiinfinitesimal,Raoul}, with an additional spatial structure, the convergence of the model towards the Kirkpatrick-Barton model when the reproduction rate is large. In all those cases, the reproduction term is assumed to be constant.  With the formalism of \eqref{eq:modUcst},  at the population scale, mating and birth should be positively correlated, which should  lead to considering the following variation on the infinitesimal operator, which acts upon the phenotype $x \in \R$  (for simplicity, we take $n=1$ here for the dimension of the phenotype):
\begin{multline}\label{sexop}
\mathcal{S} (f)(x) := \\ \dfrac{1}{\sigma \sqrt{\pi}}   \iint_{\mathbb{R}^{2}}  \exp\left[ -  \dfrac{1}{\sigma^2} \left( x - \dfrac{ x_1 +   x_2}2 \right)^2 \right ]  b(x_1) f( x_1)\dfrac{  \omega(x_2) f(  x_2)}{ \int_{\R}  \omega(x_2') f(  x_2')\, d  x_2'}\, d  x_1 d  x_2.
\end{multline}
We try to explain this operator as follows. It describes how an offspring with trait $x$ appears in the population. First, an individual $x_1$ rings a birth clock, at a rate given by its trait and the distribution of birth events $b$, as in \eqref{bhetero}. Next, this individual mates with a second parent $x_2$, chosen according to the weight  $\omega$.  Then, the trait of the offspring is drawn from the normal law $\mathcal{N} \left(\frac{x_1 + x_2}{2}, \frac{\sigma^2}2 \right)$.

As the birth rate of individuals seems a decisive factor in being chosen as a second parent, a reasonable choice would be $\omega = b$ in the formula above.  Again, to the best of our knowledge, no mathematical tools have been developed to tackle the issues we raise in this article with this new operator. We can mention the recent work \cite{raoul2021exponential} about similar operators.

A new trade-off, similar to the one discussed in this article, can also arise with the operator \eqref{sexop}. Indeed, coupled with a  selection term, as in \eqref{bhetero} for instance, a trade-off between birth and survival can appear if $b$ (or $\omega$) and $d$ have different optima. In such case, a natural question is whether the effects highlighted in this paper for asexual reproduction are still present, and when, with sexual reproduction.
Of course, a third factor in the trade-off is also present, through the weight of the choice of the second parent via the function $\omega$. If an external factor favours a second parent around a third optimum, then the effect it has on the population should also be taken into account.  
The relevance of such a model in an individual based setting, as in \Cref{sec:causes} is also an open question to this day for the operator \eqref{sexop}. 
 With the assumption $\omega = b$,  the roles of first and second parents are symmetric in the operator \eqref{sexop}, and an investigation of the balance between birth and survival could be carried out without  additional assumptions.

\section*{Acknowledgements}
This work was supported by the French Agence Nationale de la Recherche (ANR-18-CE45-0019 `RESISTE'). We thank Guillaume Martin for many fruitful discussions.

\section*{Declarations of interest: none} 

\section*{References}

\bibliography{biblio_lionel,test}

\clearpage

\centerline{\textbf{APPENDICES}}

\setcounter{figure}{0}
\setcounter{section}{0}
\renewcommand{\thesection}{\Alph{section}}
\renewcommand{\thesubsection}{\thesection.\arabic{subsection}}
\renewcommand{\thefigure}{\thesection.\arabic{figure}}

\section{Proofs \label{app:proof}}


\subsection{Proof of Proposition~\ref{prop:cvg_non_overlapping}}\label{proof firstapp}

	For $ \phi : \Omega \to \R $ measurable and bounded, let
	\begin{equation*}
		P\phi(\x) = (1-U) \phi(\x_i) + U \int_{\Omega} \phi(\y) \rho_K(\x_i,\y)d\y.
	\end{equation*}
	
	\begin{lem} \label{lemma:martingale}
		For $t \in \N$, for any $\phi : \Omega \to \R$ measurable and bounded,
		\begin{equation*}
		\langle \nu^K_t, \phi \rangle = \langle \nu^K_0, \phi \rangle + \sum_{s=1}^{t} \langle \nu^K_s, w_K e^{-c_K K \langle \nu^K_s, 1 \rangle} P \phi - \phi \rangle + M_t^K(\phi),
		\end{equation*}
		where $(M_t^K(\phi), t \geq 0)$ is a local martingale with quadratic variation
		\begin{equation*}
		\sum_{s=1}^{t} \left( \frac{1}{K} \langle \nu^K_s, w_K P(\phi^2) \rangle e^{-c_K K \langle \nu^K_s, 1 \rangle} + \langle \nu^K_s, w_K e^{-c_K K \langle \nu^K_s, 1 \rangle} P\phi - \phi \rangle^2 \right).
		\end{equation*}
	\end{lem}
	
	\begin{proof}
		From the definition of the model, for any $\phi: \Omega \to \R$ measurable and bounded,
		\begin{linenomath}
\begin{align*}
     \mathbb{E}\left[ \left. \langle \nu^K_{t+1}, \phi \rangle \right| \nu^K_t \right] &= \frac{1}{K} \sum_{i=1}^{N_t} w_K(\x_i)\, e^{-c_K N_t} \left\lbrace (1-U)\phi(\x_i) + U \int_\Omega \phi(\y) \rho_K(\x_i, d\y) \right\rbrace \\
		&= \langle \nu^K_t, w_K e^{-c_K K \langle \nu^K_t, 1 \rangle} P\phi \rangle.
\end{align*}
\end{linenomath}Hence
		\begin{equation} \label{first_moment_increment}
			\mathbb{E}\left[ \left. \langle \nu^K_{t+1}, \phi \rangle - \langle \nu^K_t, \phi \rangle \right| \nu^K_t \right] = \langle \nu^K_t, w_K e^{-c_K K \langle \nu^K_t, 1 \rangle} P\phi - \phi \rangle.
		\end{equation}
		We now wish to compute
		\begin{equation*}
		\mathbb{E}\left[ \left. \left( \langle \nu^K_{t+1},\phi \rangle - \langle \nu^K_t,\phi \rangle \right)^2 \right| \nu^K_t \right].
		\end{equation*}
		To do this, let $\nu^K_t = \frac{1}{K} \sum_{i=1}^{N} \delta_{\x_i}$ and let $N_i$ be the number of offspring of individual $i$ at time $t+1$ and let $(Y_{i,j}, 1 \leq j \leq N_i)$ denote their types.
		From the definition of the model, $N_i$ is a Poisson random variable with parameter $w_K(\x_i) e^{-c_K N}$ and the $(Y_{i,j}, j \geq 1)$ are i.i.d. with
		\begin{linenomath}
\begin{align*}
     \mathbb{E}[\phi(Y_{i,j}) \,|\, \x_i] = P\phi(\x_i), && \mathbb{V}[\phi(Y_{i,j}) \,|\, \x_i] = P(\phi^2)(\x_i) - (P\phi(\x_i))^2.
\end{align*}
\end{linenomath}
		Then we write
		\begin{multline*}
		\langle \nu^K_{t+1},\phi \rangle - \langle \nu^K_t, \phi \rangle = \frac{1}{K} \sum_{i=1}^{N} \left( \sum_{j=1}^{N_i} \left( \phi(Y_{i,j}) - P\phi(\x_i) \right) \right) \\+ \frac{1}{K} \sum_{i=1}^{N} \left( N_i - w_K(\x_i) e^{-c_K N} \right) P\phi(\x_i)  + \langle \nu^N_t, w_K e^{-c_K N} P\phi - \phi \rangle.
		\end{multline*}
		Since the third term depends only on $\nu^N_t$ and the first two terms are uncorrelated,
		\begin{multline*}
		\mathbb{E}\left[ \left. \left( \langle \nu^K_{t+1},\phi \rangle - \langle \nu^K_t,\phi \rangle \right)^2 \right| \nu^K_t \right] = \frac{1}{K^2} \sum_{i=1}^{N} w_K(\x_i) e^{-c_K N} \left(P(\phi^2)(\x_i) - (P\phi(\x_i))^2\right)\\ + \frac{1}{K^2} \sum_{i=1}^{N} w_K(\x_i) e^{-c_K N} (P\phi(\x_i))^2 + \langle \nu^K_t, w_K e^{-c_K N} P\phi - \phi \rangle^2.
		\end{multline*}
		Rearranging, we arrive at
		\begin{multline} \label{second_moment_increment}
		\mathbb{E}\left[ \left. \left( \langle \nu^K_{t+1},\phi \rangle - \langle \nu^K_t,\phi \rangle \right)^2 \right| \nu^K_t \right] \\ = \frac{1}{K} \langle \nu^K_t, w_K P(\phi^2) \rangle e^{-c_K K \langle \nu^K_t, 1 \rangle} + \langle \nu^K_t, w_K e^{- c_K K \langle \nu^K_t, 1 \rangle}  P\phi - \phi \rangle^2.
		\end{multline}
		This concludes the proof of the lemma.
	\end{proof}

    Note that, setting
    \begin{equation*}
    \mutx_K\phi(\x) = \frac{U}{\varepsilon_K} \int_\Omega (\phi(\y)-\phi(\x))\rho_K(\x,\y)d\y,
    \end{equation*}
    equation \eqref{first_moment_increment} can also be written
	\begin{multline*}
		\mathbb{E}\left[ \left. \langle \nu^K_{t+1}, \phi \rangle - \langle \nu^K_t, \phi \rangle \right| \nu^K_t \right] \\ = \langle \nu^K_t, w_K\, e^{-c_K K \langle \nu^K_t, 1 \rangle}\, \varepsilon_K \mutx_K \phi + (w_K\, e^{-c_K K \langle \nu^K_t, 1 \rangle}-1)\, \phi \rangle.
	\end{multline*}
		Using Assumption (A1), we then have
		\begin{equation*}
		\mathbb{E}\left[ \left. \langle \nu^K_{t+1}, \phi \rangle - \langle \nu^K_t, \phi \rangle \right| \nu^K_t \right] = \varepsilon_K \langle \nu^K_t, \mutx_K\phi + (m - c \langle \nu^K_t, 1 \rangle) \phi \rangle + o(\varepsilon_K \langle \nu^K_t, 1\rangle).
		\end{equation*}
		We then note that, in the case of Assumption (FE), $\mutx_K\phi = \mutx\phi$, while in the case of Assumption (SE), by a Talyor expansion, for any $\phi \in C^2_0(\Omega)$,
		\begin{equation*}
		\mutx_K \phi(\x) = \frac{\lambda U}{2} \Delta \phi(\x) + o(1),
		\end{equation*}
		uniformly in $\x\in \Omega$.
		Finally, note that the first term in \eqref{second_moment_increment} is of the order of $ 1/K $ while the second term is of the order of $ \varepsilon_K^2 $.
	
	For $N \geq 1$, define a stopping time $ \tau^K_N $ by
	\begin{equation*}
	\tau^K_N = \inf \lbrace t \geq 0 : \langle \nu^K_t, 1 \rangle \geq N \rangle.
	\end{equation*}
	
	\begin{lem} \label{lemma:cvg_martingale}
		For any fixed $ N \geq 1 $ and $ T > 0 $, for any $ \phi \in C^2_0(\Omega) $,
		\begin{equation*}
			\sup_{0 \leq t \leq \lfloor T/\varepsilon_K \rfloor} | M^K_{t \wedge \tau^K_N}(\phi) | \longrightarrow 0,
		\end{equation*}
		in probability as $ K \to \infty $.
	\end{lem}

    \begin{proof}
    	By Doob's martingale inequality,
    	\begin{multline} \label{doob}
    		\mathbb{E}\left[ \sup_{0 \leq t \leq \lfloor T/\varepsilon_K \rfloor} | M^K_{t \wedge \tau^K_N}(\phi) |^2 \right] \\
    		\begin{aligned}
    		&\leq 4\, \mathbb{E}\left[ |M^K_{\lfloor T/\varepsilon_K \rfloor \wedge \tau^K_N}(\phi)|^2 \right] \\
    		&\leq 4\, \mathbb{E}\left[ \sum_{s=1}^{\lfloor T/\varepsilon_K \rfloor \wedge \tau^K_N} \left\lbrace \frac{1}{K} \langle \nu^K_s, w_K P\phi^2 \rangle + \langle \nu^K_s, w_K e^{-c \varepsilon_K \langle \nu^K_s, 1 \rangle} P\phi-\phi\rangle^2 \right\rbrace \right].
    		\end{aligned}
    	\end{multline}
    	Clearly, for $ 0 \leq s \leq \lfloor T/\varepsilon_K \rfloor \wedge \tau^K_N $,
    	\begin{equation} \label{bound_qvar_1}
    		|\langle \nu^K_s, w_K P\phi^2 \rangle| \leq C \| \phi \|_\infty^2 N.
    	\end{equation}
    	We then note that there exists a function $ r_K : \R \to \R $ such that, for all $ x \in \R $,
    	\begin{equation*}
    	e^{\varepsilon_K x} = 1 + \varepsilon_K x\, e^{\varepsilon_K r_K(x)}, \qquad \text{ and } \qquad |r_K(x)| \leq |x|.
    	\end{equation*}
    	With this notation,
    	\begin{equation*}
    	w_K(\x) e^{-c_K K \langle \nu^K_s, 1 \rangle} - 1 = \varepsilon_K m(\x) e^{\varepsilon_K r_K(m(\x)) - c \,\varepsilon_K \langle \nu^K_s, 1 \rangle} - c\, \varepsilon_K \langle \nu^K_s, 1 \rangle e^{\varepsilon_K r_K(-c\langle \nu^K_s, 1 \rangle)}.
    	\end{equation*}
    	Hence, using the fact that $r_K(\x)$ has the same sign as $x$,
    	\begin{equation} \label{bound_qvar_2}
    		|\langle \nu^K_s, (w_K e^{-c\varepsilon_K \langle \nu^K_s, 1 \rangle} - 1) P\phi \rangle| \leq C \|\phi\|_\infty \,\varepsilon_K (N + N^2).
    	\end{equation}
    	Finally, $ P\phi - \phi = \varepsilon_K \mutx_K\phi $, and, for $ \phi \in C^2_0(\Omega) $, under either Assumption (FE) or (SE),
    	\begin{equation*}
    		\sup_{K>0}\| \mutx_K \phi \|_\infty \leq C_\phi,
    	\end{equation*}
    	we thus have
    	\begin{equation} \label{bound_qvar_3}
    		|\langle \nu^K_s, P\phi-\phi \rangle| \leq C_\phi\, \varepsilon_K N.
    	\end{equation}
    	Plugging \eqref{bound_qvar_1}, \eqref{bound_qvar_2} and \eqref{bound_qvar_3} in \eqref{doob}, we obtain
    	\begin{equation*}
    		\mathbb{E}\left[ \sup_{0 \leq t \leq \lfloor T/\varepsilon_K \rfloor} | M^K_{t \wedge \tau^K_N}(\phi) |^2 \right] \leq C T \left( \frac{N}{\varepsilon_K K} + \varepsilon_K^2 (N+N^2)^2 \right).
    	\end{equation*}
    	Since the right-hand-side tends to zero as $ K \to \infty $, this concludes the proof of the lemma.
    \end{proof}
		
		\begin{lem} \label{lemma:total_size}
			Fix $T > 0$, and let
			\begin{equation*}
				X_K = \sup_{0 \leq t \leq \lfloor T/\varepsilon_K \rfloor} \langle \nu^K_t, 1 \rangle.
			\end{equation*}
			Then $ (X_K, K > 0) $ is tight in $\R_+$.
			Moreover, for any $ \delta > 0 $, $ N $ can be chosen such that
			\begin{equation*}
				\limsup_{K \to \infty} \mathbb{P}(\tau^K_N \leq \lfloor T / \varepsilon_K \rfloor) \leq \delta.
			\end{equation*}
		\end{lem}
	
	\begin{proof}
		Looking at the statement of Lemma~\ref{lemma:martingale}, we note that $ \mutx_K 1 = 0 $ and that
		\begin{linenomath}
\begin{align*}
     			w_K(\x) e^{-c_K K \langle \nu^K_s, 1 \rangle} - 1 &= \left( e^{\varepsilon_K m(\x)} - 1 \right) e^{-c \varepsilon_K \langle \nu^K_s, 1 \rangle} + e^{-c \varepsilon_K \langle \nu^K_s, 1 \rangle} - 1 \\
			&\leq C \varepsilon_K,
\end{align*}
\end{linenomath}
		for some constant $ C > 0 $, using the fact that $m$ is bounded.
		As a consequence,
		\begin{linenomath}
\begin{align*}
     \langle \nu^K_{t \wedge \tau^K_N}, 1 \rangle &\leq \langle \nu^K_0, 1 \rangle +  \sum_{s=1}^{t \wedge \tau^K_N} \varepsilon_K C \langle \nu^K_s, 1 \rangle + M_{t \wedge \tau^K_N}^K(1) \\
			&\leq \langle \nu^K_0, 1 \rangle +  \sum_{s=1}^{t} \varepsilon_K C \langle \nu^K_{s \wedge \tau^K_N}, 1 \rangle + M_{t \wedge \tau^K_N}^K(1).
\end{align*}
\end{linenomath}
		By Gronwall's inequality, we obtain
		\begin{equation*}
			\langle \nu^K_{t \wedge \tau^K_N}, 1 \rangle \leq \left( \langle \nu^K_0, 1 \rangle + \sup_{0 \leq s \leq t} M^K_{s \wedge \tau^K_N}(1) \right) e^{C \varepsilon_K t}.
		\end{equation*}
		Hence,
		\begin{equation*}
		\sup_{0 \leq t \leq \lfloor T/\varepsilon_K \rfloor}\langle \nu^K_{t \wedge \tau^K_N}, 1 \rangle \leq \left( \langle \nu^K_0, 1 \rangle + \sup_{0 \leq t \leq  \lfloor T/\varepsilon_K \rfloor} M^K_{t \wedge \tau^K_N}(1) \right) e^{C T}.
		\end{equation*}
		As a result,
		\begin{linenomath}
\begin{align*}
     \mathbb{P}(X_K \geq N) &= \mathbb{P}(\tau^K_N \leq \lfloor T / \varepsilon_K \rfloor) \\
			&\leq \mathbb{P}\left( \left[ \langle \nu^K_0, 1 \rangle + \sup_{0 \leq t \leq  \lfloor T/\varepsilon_K \rfloor} M^K_{t \wedge \tau^K_N}(1) \right] e^{C T} \geq N \right) \\
			&\leq \mathbb{P}\left(  \langle \nu^K_0, 1 \rangle \geq \frac{1}{2}N e^{-CT} \right) + \mathbb{P}\left( \sup_{0 \leq t \leq  \lfloor T/\varepsilon_K \rfloor} M^K_{t \wedge \tau^K_N}(1) \geq \frac{1}{2} N e^{-CT} \right).
\end{align*}
\end{linenomath}
		Since $ \langle \nu^K_0, 1 \rangle $ is tight, for any $ \delta > 0 $ we can choose $ N $ large enough such that
		\begin{equation*}
			\limsup_{K \to \infty} \mathbb{P}\left(  \langle \nu^K_0, 1 \rangle \geq \frac{1}{2}N e^{-CT} \right) \leq \delta.
		\end{equation*}
		In addition, by Lemma~\ref{lemma:cvg_martingale}, for any $ N \geq 1 $,
		\begin{equation*}
			\limsup_{K \to \infty}  \mathbb{P}\left( \sup_{0 \leq t \leq  \lfloor T/\varepsilon_K \rfloor} M^K_{t \wedge \tau^K_N}(1) \geq \frac{1}{2} N e^{-CT} \right) = 0.
		\end{equation*}
		Hence we can choose $ N $ large enough such that
		\begin{equation*}
			\limsup_{K \to \infty} \mathbb{P}(X_K \geq N) \leq \delta,
		\end{equation*}
		concluding the proof.
	\end{proof}

    \begin{lem} \label{lemma:tightness}
        Let $\overline{\Omega}$ denote the closure of $\Omega$.
    	Then, for any $ T > 0$, the sequence of processes
    	\begin{equation*}
    		\left( \nu^K_{\lfloor t / \varepsilon_K \rfloor}, t \in [0,T] \right), \quad K > 0,
    	\end{equation*}
    	is C-tight in $D([0,T], M_F(\overline{\Omega}))$.
    \end{lem}

    \begin{proof}
    Since $\overline{\Omega}$ is compact, Lemma~\ref{lemma:total_size} implies that the compact containment condition of Theorem~3.9.1 in \citep{ethier_markov_1986} holds.
    It remains to show that, for a dense subset $H$ of the set of bounded and continuous functions on $M_K(\overline{\Omega})$ (in the topology of uniform convergence on compact sets), $(h(\nu^K_t), t \in [0,T])$ is C-tight for every $h \in H$.
    We shall take
    \begin{equation*}
        H = \lbrace \nu \mapsto f(\langle \nu, \phi \rangle), \phi \in C^2_0(\Omega), f \in C^2(\R) \rbrace.
    \end{equation*}
    
    	By Lemma~\ref{lemma:total_size}, it is sufficient to show that $(f(\langle \nu^K_{\lfloor t/\varepsilon_K\rfloor \wedge \tau^K_N}, \phi \rangle), t \in [0,T])$ is tight for any $ N $ large enough.
    	Now, using \eqref{bound_qvar_2} and \eqref{bound_qvar_3}, for $0 \leq s \leq \lfloor T / \varepsilon_K \rfloor \wedge \tau^K_N$,
    	\begin{equation*}
    		|\langle \nu^K_s, w_K e^{-c_K K \langle \nu^K_s, 1 \rangle} P\phi - \phi \rangle| \leq C_{\phi,N}\, \varepsilon_K,
    	\end{equation*}
    	for some constant $ C_{\phi,N} > 0 $ depending only on $\phi$ and $N$.
    	As a result, if $ w_\theta(g) $ denotes the modulus of continuity of $g : [0,T] \to \R$, i.e.,
    	\begin{equation*}
    		w_\theta(g) = \sup_{\underset{0 \leq s \leq t \leq T}{|t-s| \leq \theta}} |g(t)-g(s)|,
    	\end{equation*}
    	we obtain
    	\begin{equation*}
    		w_\theta\left( f(\langle \nu^K_{\lfloor \cdot / \varepsilon_K \rfloor \wedge \tau^K_N}, \phi \rangle) \right) \leq \sup_{|x| \leq N} |f'(x)| \left( C_{\phi,N}\, \theta + 2 \sup_{0 \leq t \leq \lfloor T / \varepsilon_K \rfloor} |M^K_{t \wedge \tau_N^K}(\phi)| \right).
    	\end{equation*}
    	Hence, by Lemma~\ref{lemma:cvg_martingale}, for any $ \delta > 0 $ and $ \varepsilon > 0 $, there exists $ \theta > 0 $ such that
    	\begin{equation*}
    		\limsup_{K \to \infty}\, \mathbb{P}\left( w_\theta\left( f(\langle \nu^K_{\lfloor \cdot / \varepsilon_K \rfloor \wedge \tau^K_N}, \phi \rangle) \right) > \delta \right) \leq \varepsilon.
    	\end{equation*}
    	Combined with Lemma~\ref{lemma:total_size}, this shows that $( f(\langle \nu^K_{\lfloor t/\varepsilon_K\rfloor \wedge \tau^K_N}, \phi \rangle), t \in [0,T])$ is C-tight for any $ \phi \in C^2_0(\Omega) $, $f \in C^2(\R)$ and $ N > 0 $ (see for example \citep[Proposition~VI.3.26]{jacod_limit_2003}), and the result is proved.
    \end{proof}

    We can now conclude the proof of the main result.
    
    \begin{proof}[Proof of Proposition~\ref{prop:cvg_non_overlapping}]
        \Rd The fact that \eqref{limit_f_discrete} defines a unique function $(f_t, t \in [0,T])$ taking values in $M_F(\Omega)$ is proved in \cite{ChaFer08} (in the proof of Theorem~4.3). \Bk
    	Consider then a converging subsequence, still denoted \[\left(\nu^K_{\lfloor t / \varepsilon_K \rfloor}, t \in [0,T] \right)\] and let $(f_t, t \in [0,T])$ be its limit.
    	Since the sequence is C-tight, $t \mapsto f_t$ is continuous and the convergence holds uniformly on $ [0,T] $.
    	
    	The result will be proved if we show that $f_t$ solves \eqref{limit_f_discrete}.
    	Let $ \phi \in C^2(\Omega) $.
    	By Lemma~\ref{lemma:cvg_martingale},
    	\begin{equation*}
    		\sup_{0 \leq t \leq  \lfloor T/\varepsilon_K \rfloor} \left|  \langle \nu^K_{t\wedge \tau^K_N}, \phi \rangle - \langle \nu^K_0, \phi \rangle - \sum_{s=1}^{t\wedge \tau^K_N} \langle \nu^K_s, w_K e^{-c_K K \langle \nu^K_s, 1 \rangle} P\phi - \phi \rangle \right| \longrightarrow 0
    	\end{equation*}
    	in probability as $ K \to \infty $.
    	In addition,
    	\begin{multline*}
    		w_K(\x) e^{-c_K K \langle \nu^K_s, 1 \rangle} P\phi(\x) - \phi(\x) - \varepsilon_K \left( \mutx\phi(\x) - (m(\x) - c \langle \nu^K_s, 1 \rangle) \phi(\x) \right) \\= \left( e^{\varepsilon_K (m(\x)- c \langle \nu^K_s, 1 \rangle)} - 1 - \varepsilon_K (m(\x) - c \langle \nu^K_s, 1 \rangle) \right) P\phi(\x) \\+ P\phi(\x) - \phi(\x) - \varepsilon_K \mutx\phi(\x) \\+ \varepsilon_K (m(\x) - c \langle \nu^K_s, 1 \rangle) (P\phi(\x) - \phi(\x)).
    	\end{multline*}
    	Hence, for $ 0 \leq s \leq \lfloor T / \varepsilon_K \rfloor \wedge \tau^K_N $ and $ \phi \in C^2_0(\Omega) $,
    	\begin{multline*}
    		\left|w_K(\x) e^{-c_K K \langle \nu^K_s, 1 \rangle} P\phi(\x) - \phi(\x) - \varepsilon_K \left( \mutx\phi(\x) - (m(\x) - c \langle \nu^K_s, 1 \rangle) \phi(\x) \right)\right| \\
    		\begin{aligned}
    		&\leq C \varepsilon_K^2 + \varepsilon_K \left( \mutx_K \phi(\x) - \mutx\phi(\x) \right) \\ 
    		&\leq C' \varepsilon_K^2.
    		\end{aligned}
    	\end{multline*}
    	As a result,
    	\begin{multline*}
    		\sup_{0 \leq t \leq  \lfloor T/\varepsilon_K \rfloor} \left| \sum_{s=1}^{t \wedge \tau^K_N} \langle \nu^K_s, w_K e^{-c_K K \langle \nu^K_s, 1 \rangle} P\phi - \phi \rangle \right. \\ \left. - \sum_{s=1}^{t \wedge \tau^K_N} \varepsilon_K \langle \nu^K_s, \mutx\phi + (m-c\langle \nu^K_s,1 \rangle) \phi \rangle \right| \leq C' T \varepsilon_K.
    	\end{multline*}
    	Hence, on the event $ \lbrace \tau^K_N > \lfloor T / \varepsilon_K \rfloor \rbrace$,
    	\begin{multline*}
    		\sup_{0 \leq t \leq T} \left| \langle f_t, \phi \rangle - \langle f_0, \phi \rangle - \int_{0}^{t} \langle f_s, \mutx\phi - (m - c \langle f_s, 1 \rangle) \phi \rangle ds \right|\\ \leq 2 \sup_{0 \leq t \leq T} | \langle \nu^K_{\lfloor t/ \varepsilon_K \rfloor}, \phi \rangle - \langle f_t, \phi \rangle | + C' T \varepsilon_K + \sup_{0 \leq t \leq  \lfloor T/\varepsilon_K \rfloor} \left| M^K_{t \wedge \tau^K_N}(\phi) \right| \\+ T \sup_{0 \leq s \leq T} \left| \langle \nu^K_s, \mutx\phi + (m-c\langle \nu^K_s, 1 \rangle) \phi \rangle - \langle f_s, \mutx\phi + (m-c\langle f_s, 1 \rangle) \phi \rangle \right|.
    	\end{multline*}
    	Combined with Lemma~\ref{lemma:total_size} and the (uniform) convergence of $\nu^K_{\lfloor \cdot / \varepsilon_K \rfloor}$ to $f$, this shows that, for any $ \varepsilon > 0 $,
    	\begin{equation*}
    		\mathbb{P} \left( \sup_{0 \leq t \leq T} \left| \langle f_t, \phi \rangle - \langle f_0, \phi \rangle - \int_{0}^{t} \langle f_s, \mutx\phi - (m - c \langle f_s, 1 \rangle) \phi \rangle ds \right| > \varepsilon \right) = 0.
    	\end{equation*}
    	It follows that $(f_t, t \in [0,T])$ solves \eqref{limit_f_discrete}, hence $(\nu^N_t, t \in [0,T])$ converges in distribution, and in probability, to $(f_t, t \in [0,T])$ in $D([0,T], M_F(\overline{\Omega})$.
    	Since in fact $f_t \in M_F(\Omega)$ for any $t \geq 0$, this concludes the proof of the result.
    \end{proof}

\subsection{Proof of Proposition~\ref{prop:initial_bias}}\label{proof initial}

Multiplying equation \eqref{eq:main_model} by $x_1$, integrating over $\x \in \Omega$ and evaluating at $t=0$, we get
	\begin{eqnarray*}
 \xxb'(0)&=&\p_t \left(  \int_\Omega x_1 q(t,\x) d\x\right) (t=0)  \\ 
&=&	D\, \int_\Omega x_1 \Delta (bq_0)(\x) d\x + \int_\Omega x_1 q_0(\x) m(\x) d\x - \overline m(0) \int_\Omega x_1 q_0(\x) d\x .
	\end{eqnarray*}
From Green formula we infer
	\begin{align*}
	\int_\Omega x_1 \Delta   (bq_0)(\x) d\x  & = \int_{\partial \Omega} x_1\,  \nabla (bq_0) \cdot \nub(\x) \, ds - \int_{\partial \Omega} (bq_0)(\x)\, \vec{e_1}\cdot \nub(\x) \, ds, \\
	& = 0 
	\end{align*}
 since $q_0$ is compactly supported in $\Omega$. Moreover, since $q_0$ and $m$ both satisfy $m(\iota(\x))=m(\x)$, $q_0(\iota(\x))=q(\x)$,  \begin{equation*}
	\int_\Omega x_1 q_0(\x) m(\x) d\x = \int_\Omega x_1 q_0(\x)  d\x = 0.
	\end{equation*}
	This shows that $\xxb'(0)=0$.

	We next turn to the second derivative $\xxb ''(0)$. We differentiate equation \eqref{eq:main_model} with respect to time, multiply by $x_1$, integrate over $\x \in \Omega$ and evaluate at $t=0$ to reach
	\begin{multline}
\xxb ''(0)=	\p_{tt}\left(  \int_\Omega x_1 q(t,\x) d\x\right) (t=0) 
=	D \int_\Omega x_1 \Delta (b\p_t q(0,\x)) d\x + \int_\Omega x_1 \p_t q(0,\x) m(\x) d\x 
\\- \overline m(0) \int_\Omega x_1 \p_t q(0,\x) d\x  - \overline m'(0) \int_\Omega x_1 q_0(\x) d\x\nonumber.
\end{multline}
 From the above computation, this reduces to 
\begin{equation*}
\xxb ''(0)= D \int_\Omega x_1 \Delta (b\p_t  q(0,\x)) d\x + \int_\Omega x_1 \p_t q(0,\x) m(\x) d\x.
\end{equation*}
Moreover, since $\p_t q(0,\x)$ is also compactly supported (this follows from equation~\eqref{eq:main_model}), Green formula yields \begin{equation*}\label{eq:sign inter3}
	\int_\Omega x_1 \Delta (b\p_t q(0,\x)) d\x = 0,
	\end{equation*}
	and we are left with
\begin{equation}\label{eq:sign inter}
\xxb ''(0)= \int_\Omega x_1 \p_t q(0,\x) m(\x) d\x.
\end{equation}	
	We multiply equation \eqref{eq:main_model} by $x_1 m(\x)$, integrate over $\x \in \Omega$ and evaluate at $t=0$ to obtain
	\begin{eqnarray*}
	\int_\Omega x_1  \p_t q(0,\x) m(\x) d\x &=& D\int_\Omega \Delta (b q_0) (\x) x_1 m(\x) d\x + \int_\Omega x_1  q_0(\x) m(\x)^2 d\x \\ 
	&&-\overline m(0) \int_\Omega x_1  q_0(\x) m(\x)d\x.
	\end{eqnarray*}
	By symmetry, the last two terms vanish, and another Green formula leads to
	\begin{equation}\label{green3}
	\int_\Omega x_1  \p_t q(0,\x) m(\x) d\x  = D\int_\Omega (b q_0) (\x) \Delta( x_1 m(\x) ) d\x.
	\end{equation}
	Then, we observe that
		\begin{align*}
	    \int_{\Omega\cap \{x_1<0 \}} b(\x)\,  q_0(\x) \Delta( x_1 m(\x) ) d\x  & =  -\int_{\Omega\cap \{x_1>0 \}} s(\x) q_0 (\x)  \Delta( x_1 m(\x) ) d\x,
	    \end{align*}
	    as $q_0$ and $m$ are symmetric about $\{x_1=0\}$, and from \eqref{eq:sym_b_s}. Thus,
	\begin{eqnarray*}
	    \int_\Omega (b q_0) (\x) \Delta( x_1 m(\x) ) d\x  &=&\int_{\Omega\cap \{x_1<0 \}} (b q_0) (\x) \Delta( x_1 m(\x) ) d\x \\ && \qquad \qquad  +\int_{\Omega\cap \{x_1 > 0\}} (b q_0) (\x) \Delta( x_1 m(\x) ) d\x, \\
	    &=& \int_{\Omega\cap \{x_1 > 0\}\cap K_0} (b-s)(\x) \, q_0 (\x) \, \Delta( x_1 m(\x) ) d\x,
	\end{eqnarray*}
with $K_0$ the support of $q_0$ (containing $\x_0$). From this, \eqref{eq:sign inter} and \eqref{green3}, we end up with
\[
\xxb ''(0)=D \int_{\Omega\cap \{x_1 > 0\}\cap K_0} (b-s)(\x) \, q_0 (\x) \, \Delta( x_1 m(\x) ) d\x.
\]
We know from \eqref{hyp_b} that $(b-s)(\x)> 0$ in $\Omega\cap \{x_1 > 0\}$.  As a result, if $\Delta( x_1 m(\x) )$ is nontrivial and nonnegative (nonpositive) on $K_0^+=K_0\cap\{x_1>0\}$ then $\xxb ''(0)>0$ ($\xxb ''(0)<0$ respectively). This concludes the proof of \Cref{prop:initial_bias}. \qed 

\subsection{An explicit solution of the eigenvalue problem}\label{proof explicit}
	
We assume that the dimension is $n=1$ and
	\begin{equation*}
	b(x)=\baco{l}
	2 \quad \hbox{ for } x\in  (0,a),\\
	1 \quad \hbox{ for } x\in (-a, 0),\\
	\eaco
	\end{equation*}
and we consider the eigenvalue problem \eqref{eq:statio1} with Dirichlet boundary conditions. As $b$ is discontinuous, the eigenvalue problem  must be understood in the weak sense. In particular, we have to solve
\begin{equation}
	\baco{rl}
	\ds D\,  q_{1,\infty}'' (x)& = (\mb_\infty-r-3) \, q_{1,\infty}(x),  \ x\in (-a,0), \vspace{1mm}\\
	2D\,  q_{2,\infty}'' (x) &= (\mb_\infty-r-3) \, q_{2,\infty}(x),   \ x\in (0,a),
	\eaco
	\end{equation}
	with the boundary, continuity and flux conditions:
	\begin{equation}
	\baco{l}
	q_{1,\infty}(-a)=q_{2,\infty}(a)=0,\vspace{1mm}\\
	q_{1,\infty}(0)= q_{2,\infty}(0), \ q_{1,\infty}'(0)=2 \, q_{2,\infty}'(0),
	\eaco
	\end{equation}
	the positivity conditions $ q_{1,\infty}, \, q_{2,\infty}>0$ and $\mb_\infty-r-3<0$.
	
	Set $\mu=\sqrt{2 \, D}$ and $B:=\sqrt{-\mb_\infty+r+3}/\mu$. We have
	\begin{equation} \label{eq:sol_q1_q2}
	\baco{l}
	q_{1,\infty}(x)=\ds -\frac{\sqrt{2}}{B} \, \cos\lp x\, B\, \sqrt{2}\rp \lp \tan( x\, B\, \sqrt{2}) + \tan(a\, B\, \sqrt{2})\rp, \ x\in (-a,0),  \vspace{1mm}\\
	q_{2,\infty}(x)=\ds \frac{1}{B} \, \cos\lp x\, B\, \rp \lp \tan(a\, B)-\tan( x\, B) \rp,  \ x\in (0,a).
	\eaco
	\end{equation}
	The equality $q_{1,\infty}(0)= q_{2,\infty}(0)$ thus implies:
	\begin{equation}\label{eq:egalite_q1q2}
	\sqrt{2}\, \tan(a\, B \, \sqrt{2})=-\tan(a\,B).
	\end{equation}
	The positivity of $q_{1,\infty}, \, q_{2,\infty}$ implies that $0<a\, B< \pi/2$. The equation~\eqref{eq:egalite_q1q2} thus admits a unique solution $a\, B\in (\pi/(2\,\sqrt{2}),\pi/2)$ ($a\,B\approx 1.338761890$).
	Additionally, we have:
	\[
	\ds\frac{\int_{-a}^0 q_{1,\infty}(x)\, dx}{\int_{0}^a q_{2,\infty}(x)\, dx}  =-\frac{(1- \cos( a\, B\, \sqrt{2}))\, \cos(a\, B)}{(1- \cos( a\, B))\, \cos( a\, B\, \sqrt{2})},\]and using~\eqref{eq:egalite_q1q2},
	\begin{linenomath}
\begin{align*}
    \ds\frac{\int_{-a}^0 q_{1,\infty}(x)\, dx}{\int_{0}^a q_{2,\infty}(x)\, dx}   & = \frac{1}{\sqrt{2}}\,\frac{(1- \cos( a\, B\, \sqrt{2}))\, \sin(a\, B)}{(1- \cos( a\, B))\, \sin( a\, B\, \sqrt{2})} \\ & = \frac{1}{\sqrt{2}}\frac{j(a\,B\,\sqrt{2})}{j(a\,B)}.  
\end{align*}
\end{linenomath}
	with $j(x):=(1-\cos(x))/\sin(x)$. As $j(x\,\sqrt{2})/j(x)$ is increasing on $(\pi/(2\,\sqrt{2}),\pi/2)$, we get:
	\[ \frac{1}{\sqrt{2}}\frac{j(a\,B\,\sqrt{2})}{j(a\,B)}
	\ge \frac{1}{\sqrt{2}} \frac{j(\pi/2)}{j(\pi/(2\,\sqrt{2}))}=\frac{1}{\sqrt{2}}\, \frac{1}{j(\pi/(2\,\sqrt{2}))}.
	\]
	As $1/(2\, \sqrt{2})<3/8$ and since $j$ is increasing on $(\pi/(2\,\sqrt{2}),\pi/2)$, \[j(\pi/(2\,\sqrt{2}))<j(3\, \pi/8)=1-\sqrt{2}+\sqrt{4-2\,\sqrt{2}}.\]Finally,
	\[\ds\frac{\int_{-a}^0 q_{1,\infty}(x)\, dx}{\int_{0}^a q_{2,\infty}(x)\, dx}    = \frac{1}{\sqrt{2}}\frac{j(a\,B\,\sqrt{2})}{j(a\,B)}>\frac{1}{\sqrt{2}\, j(3\, \pi/8)}= \frac{1}{2 \,\sqrt{2-\sqrt{2}}-2+\sqrt{2}}>1. \]

\Rd{}
\subsection{Model \eqref{eq:main_model} with a constant fitness function}\label{app:num_flat}
Assume that $b(\x)>0$ and $s(\x)=b(\iota(\x))$ are such that $m(\x)$ is constant. We know that the stationary distribution $q_\infty$ is (up to a multiplicative constant) the unique positive function such that the eigenvalue problem~\eqref{eq:statio1} admits a solution. Since $m$ is constant, we observe that $1/b$ satisfies these conditions, with the principal eigenvalue $m$. Thus, $q_\infty(\x) = C/b(\x)$ for some positive constant $C$ such that $q_\infty$ has integral $1$ over $\Omega$, and $\mb_\infty=m$. This means that $q_\infty$ takes larger values when $b$ is small, and shows that the high survival strategy is promoted at equilibrium.

As in Appendix~\ref{proof explicit}, we now consider the case of a flat fitness function in dimension $n=1$, over an interval $(-a,a)$. For simplicity, we take a regularised form of the function $b$ in Appendix~\ref{proof explicit}, namely
$$
b(x)=1+\frac{1+\tanh(\alpha\,x)}2,
$$
with $\alpha=40$ in our simulations. This leads to the survival function $s(x)=1+\frac{1-\tanh(\alpha\, x)}2$ and to the constant fitness function $m(x)=3-r$, for all $x\in (-a,a)$. We numerically computed the solution of the model \eqref{eq:main_model} with these assumptions and with Neumann boundary conditions at $\pm a$, as in the rest of the main text. We started with an initial condition concentrated at $x_0=0$. Notice that the results of Proposition~\ref{prop:initial_bias} cannot be applied here as $\Delta (x m) \equiv 0.$ However, we observe in Fig.~\ref{fig:rev1} (left panel) that the mean phenotype $\overline{x}(t)$ is attracted towards negative values, corresponding to the strategy where the survival function takes higher values. As expected, $q(t,\cdot)$ converges towards the stationary distribution $q_\infty(x)=1/b/\lp\int_{-a}^a 1/b \rp$ (right panel in Fig.~\ref{fig:rev1}), which also takes larger values when $x$ is negative. This shows that, again, adaptation tends to promote the high survival strategy at large times.
The Matlab codes corresponding to these simulations are available at \url{https://osf.io/g6jub/}.

\begin{figure}[h!]
\center
\includegraphics[width=0.45\textwidth]{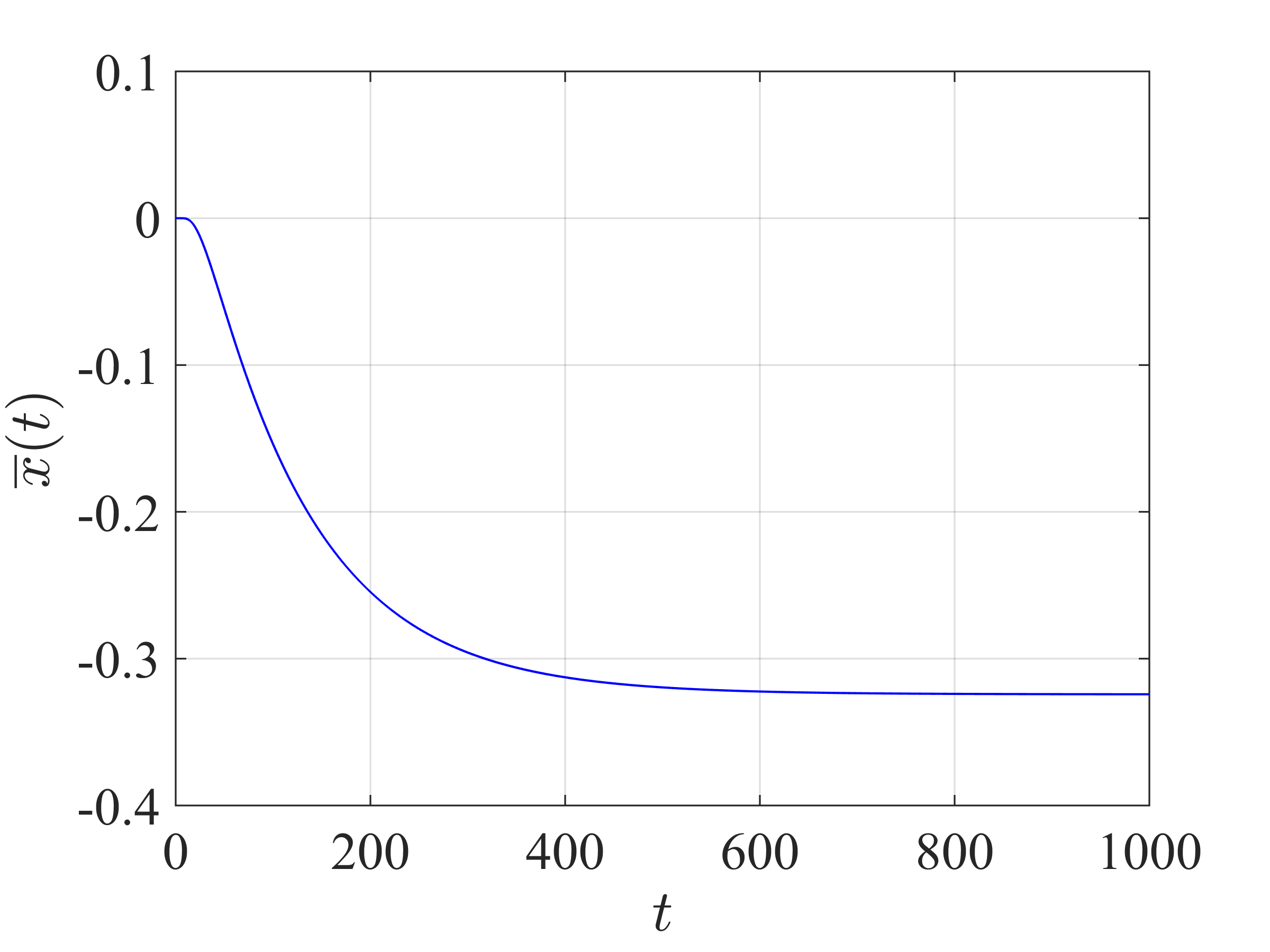}
\includegraphics[width=0.45\textwidth]{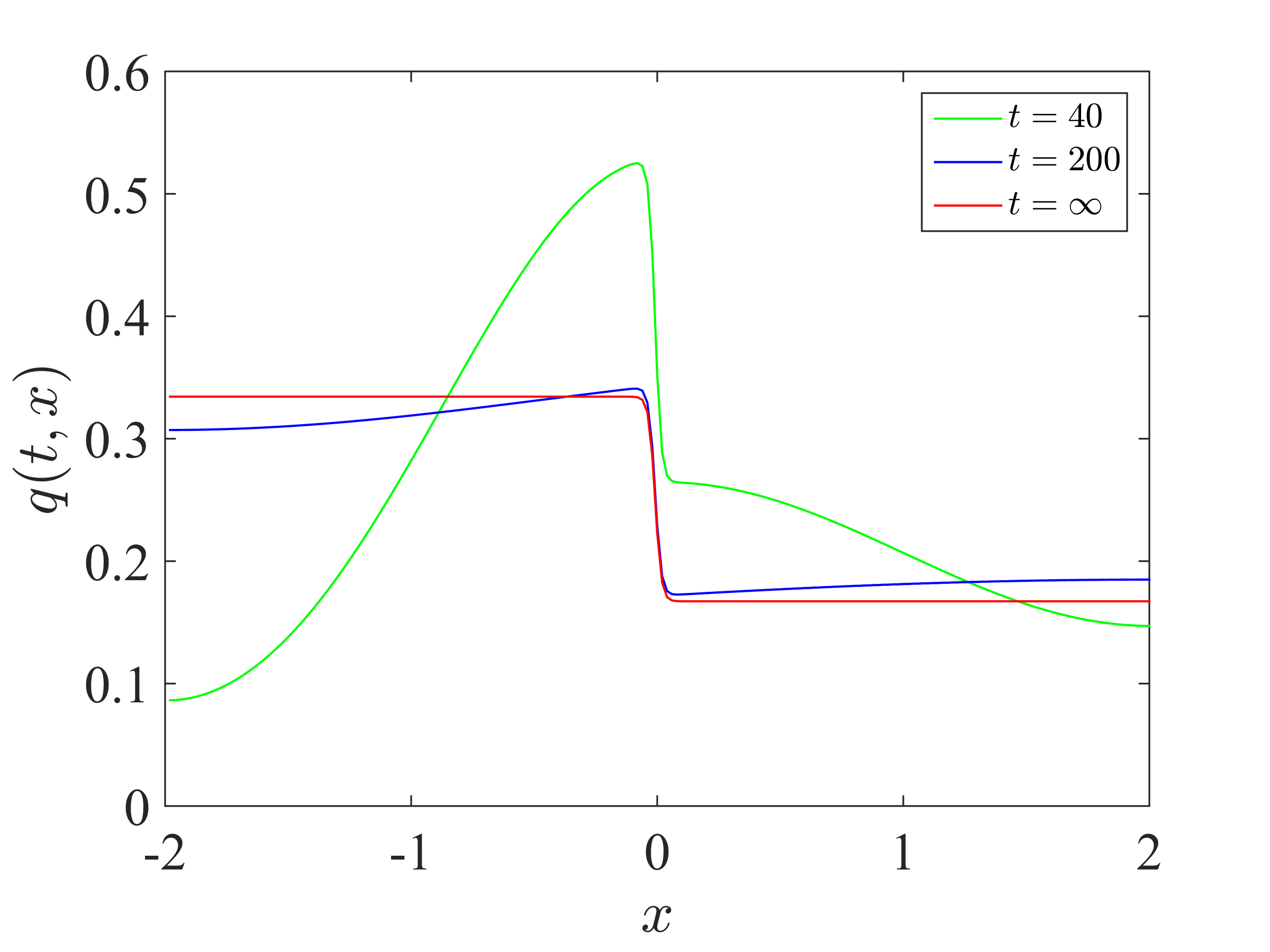}
\caption{{\bf \Rd{} Solution of the model ($\mathcal{Q}_b$) with a constant fitness function.} \Rd{} The dimension is $n=1$, $b(x)=1+(1+\tanh(\alpha\, x))/2$ and $s(x)=1+(1-\tanh(\alpha\, x))/2$ which leads to the fitness function $m(x)=3-r$. The left panel depicts the dynamics of the mean phenotype $\overline{x}(t)$ and the right panel shows the distribution $q(t,\cdot)$ at $t=40$ and $t=200$ and the equilibrium distribution $q_\infty(x)=1/b/\lp\int_{-a}^a 1/b \rp$ obtained by solving the eigenvalue problem~\eqref{eq:statio1}.  In these computations, the initial condition was concentrated at $x_0=0$. The other parameter values are $D=10^{-2}$ and $\alpha=40$. }
\label{fig:rev1}
\end{figure} 	

\Bk{}

\section{Asymmetric fitness landscapes \label{app:asym}}

In the main text, the birth and survival functions have the same height and width, so that the resulting fitness landscape $m(\x)$ is symmetric, double peaked, and both peaks have equal height. We consider here an asymmetric case, where the two peaks have different height. Namely, we consider the case:
\begin{equation}\label{eq:asym}
    \baco{l}
b(\x)=b_0+\gamma \, b_1(\x),\\
s(\x)=b_0+ b_1(\iota(\x)),
\eaco
\end{equation}
for $\gamma\neq 1$ (the case $\gamma=1$ is treated in the main text), and with a function $b_1$ with a single optimum at $\x=\Oc_b$, with  $b_1(\Oc_b)=b_{\hbox{max}}$, and which decays to $0$ away from $\Oc_b$. We recall that $\iota(\x)=\iota(x_1,x_2, ..., x_n) = (-x_1,x_2, ..., x_n).$

In this framework the fitness of the birth optimum is $m(\Oc_b)=2\, b_0 + \gamma \, b_{\hbox{max}}+\varepsilon-r$ and  the fitness of the survival optimum is  $m(\Oc_s)=2\, b_0 + b_{\hbox{max}}+\gamma \, \varepsilon-r$. In both cases, $\varepsilon:=b_1 (\Oc_s)$. We assume here that $\varepsilon\ll1$, meaning that the phenotype $\Oc_s$ has a birth rate very close to the baseline value $b_0$. Similarly, the phenotype $\Oc_b$ has a survival rate $b_0+\gamma \, \varepsilon$ close to the baseline survival rate $s_0=b_0$, see the scheme on \Cref{fig:m(x)gamma}.
\begin{figure}
\center
	\definecolor{mycolor1}{rgb}{0.00000,0.44700,0.74100}%
\def\gammag{1.3}
\def\a{6}

	\begin{tikzpicture}[scale= 0.4]
\begin{axis}[%
width=5.298in,
height=3.841in,
at={(0.889in,0.825in)},
scale only axis,
every outer x axis line/.append style={black},
every x tick label/.append style={font=\color{black},font=\Huge},
every x tick/.append style={black},
xmin=-1.3,
xmax=1.3,
xtick={-0.5,  0, 0.5},
xticklabels={$-\beta$,$0$,$\beta$},
tick align=outside,
ymin=0,
ymax = 1.5,
every outer y axis line/.append style={black},
every y tick label/.append style={font=\color{black},font=\Huge},
every y tick/.append style={black},
ytick={0,1,1.3},
yticklabels={$2b_0$,$2\, b_0 +\, b_{\hbox{max}}+ \gamma \varepsilon-r$,$2\, b_0 + \gamma \, b_{\hbox{max}}+\varepsilon-r$},
ylabel style={font=\Huge},
axis background/.style={fill=white},
axis x line*=bottom,
axis y line*=left,
xlabel style={font=\bfseries, left, at={(1,0)},font=\Huge },
xlabel={$x_1$}]
%
	
	\addplot[color =mycolor1,line width=3.0pt, samples = 100, domain = -1.3:1.3] {1.3*exp(-6*(\x-0.5)*(\x -0.5)) + exp(-6*(\x+0.5)*(\x  + 0.5))};
	
	\addplot[color =red, samples = 100, domain = -1.3:1.3] {\gammag*exp(-\a*(\x-0.5)*(\x -0.5))};
		\addplot[color =red, samples = 100, domain = -1.3:1.3] {exp(-\a*(\x+0.5)*(\x + 0.5))};
		
		\addplot [color=black, dashed, forget plot]
		table[row sep=crcr]{%
			-0.5	0\\
			-0.5	1.5\\
		};
	
		\addplot [color=black, dashed, forget plot]
		table[row sep=crcr]{%
			0.5	0\\
			0.5	1.5\\
		};
			\addplot [color=black, dashed, forget plot]
		table[row sep=crcr]{%
			-1.3	1\\
			-0.5	1\\
		};
			\addplot [color=black, dashed, forget plot]
		table[row sep=crcr]{%
			-1.3 1.3\\
			0.5	1.3\\
		};
	\end{axis}
	\end{tikzpicture}

	\label{subfigmbis}
	\caption{{\bf Schematic representation of the asymmetric fitness function $m(\x)$ considered here, along the phenotype dimension $x_1$.} Compare with Fig.~1 in the main text.  In red we pictured the functions $b + b_0-r/2$ and  $s+b_0-r/2$. Observe that the difference of fitness between the two peaks is $ (\gamma -1)(b_{\hbox{max}}- \varepsilon)$ and thus $\gamma$ tunes the asymmetry of the phenotypic landscape.}
	\label{fig:m(x)gamma}
\end{figure}
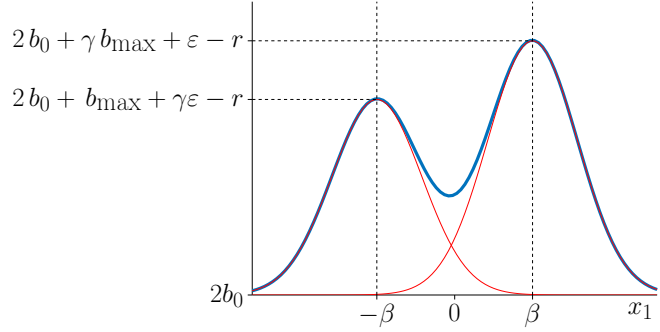

In the main text, with $\gamma=1$, we have shown that the trajectories are attracted by the survival optimum. We check here whether this remains true for asymmetric fitness functions ($\gamma\neq 1$).
In \Cref{fig:gamma}, we depict the position of the mean phenotype $\xb(t)$ (first coordinate) depending on the value of $\gamma$, at small times $(t=40)$, larges times $(t=500)$ and infinite time (in this case, we directly solve the eigenvalue problem \eqref{eq:statio2} with Comsol Multiphysics eigenvalue solver). We observe that, at small times, the trajectories are attracted by the birth optimum, whatever the value of $\gamma$, and reach positions closer to $\beta$ (the first component of $\Oc_b$) as $\gamma$ is increased. At larger times, we observe a bifurcation threshold $\gamma^*>1$ such that the trajectories are still attracted by the survival optimum when $\gamma<\gamma^*$, while they are attracted by the birth optimum for $\gamma>\gamma^*$.

We claim here that the trajectories are attracted by the survival optimum as long as the difference between the fitness peaks $m(\Oc_b)-m(\Oc_s)$ is smaller than difference between the mutation loads that would be associated with an equilibrium distribution around $\Oc_b$ vs around $\Oc_s$. To check this conjecture, we consider as in the figures of the main text, a function $b_1(\x)=\exp\left[-(\x-\Oc_b)^2/(2 \, \sigma^2)\right]$, and we assume a single-peak landscape with a unique optimum at $\Oc_b$. The corresponding fitness is  $m^b(\x)=2 \, b_0 + \gamma \, b_1(\x)-r$.  We make the weak selection approximation $\Delta(b\,q)\approx  b(\Oc_b) \, \Delta q=(b_0+1)\, \Delta q$ in the model ($\mathcal{Q}_{b}$) and  $m^b(\x)\approx 2 \, b_0 +\gamma \, (1-\|\x-\Oc_b\|^2/(2\, \sigma^2))-r$. The results in \citep{MarRoq16,HamLav20} imply that the equilibrium mean fitness is $\mb^b_\infty=2 \, b_0 +\gamma-r -n\sqrt{ 2 \,D \, (b_0+1) \, \gamma}/(2 \, \sigma)$. The mutation load is: $n\sqrt{ 2 \,D \, (b_0+1) \, \gamma}/(2 \, \sigma)$.

Now, consider a single-peak landscape with a unique optimum at $\Oc_s$, with a fitness function $m^s(\x)=2 \, b_0 +  b_1(\iota(\x))-r$ and make a weak selection approximation $\Delta(b\,q)\approx  b_0 \, \Delta q$ in the model ($\mathcal{Q}_{b}$) and  $m^s(\x)\approx 2 \, b_0 + (1-\|\x-\Oc_s\|^2/(2\, \sigma^2))-r$. This time, we get $\mb^s_\infty=2 \, b_0 +1-r-n\sqrt{ 2 \,D \, b_0}/(2 \, \sigma)$, and the mutation load is $n\sqrt{ 2 \,D \, b_0}/(2 \, \sigma)$. Finally, the difference between the fitness peaks $m(\Oc_b)-m(\Oc_s) = (\gamma-1)(b_{\hbox{max}}-\varepsilon)\approx \gamma-1 $ is smaller than difference between the corresponding mutation loads if 
\begin{equation} \label{eq:gamma}
\gamma-1 < \frac{n\sqrt{ 2 \,D}}{2 \, \sigma}(\sqrt{\gamma(b_0+1)}-\sqrt{b_0}).    
\end{equation}
With the parameter values in \Cref{fig:gamma}, this leads to $\gamma^*=1.03$ which is fully consistent with the numerical results. More generally, the above formula shows that $\gamma^*$ is an increasing function of $n$ $D$ and $1/\sigma$.

\begin{figure}[h!]
\center
\includegraphics[width=0.7\textwidth]{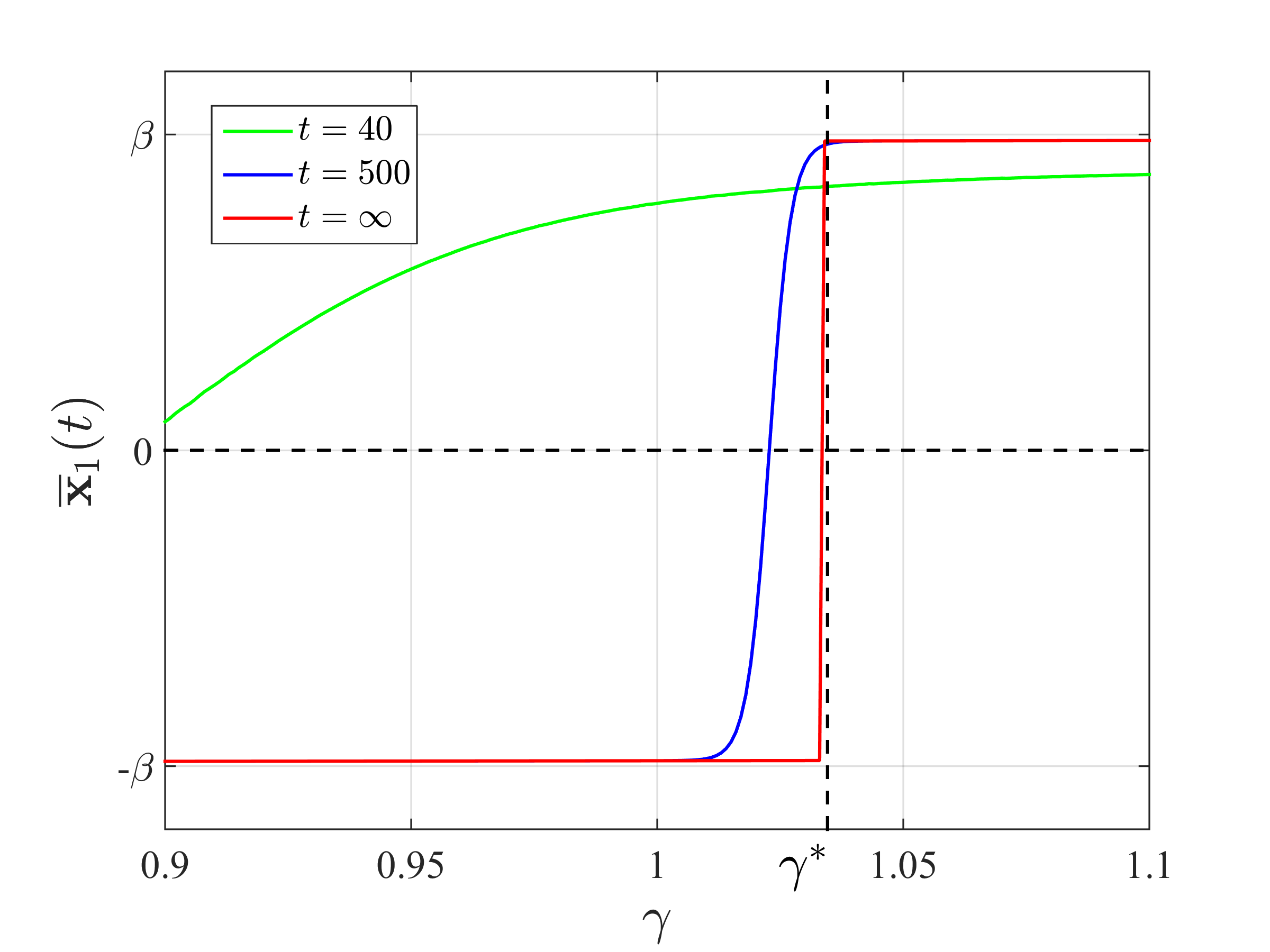}
\caption{{\bf  First component $\xb_1(t)$ of the mean phenotype depending on the asymmetry coefficient~$\gamma$: model ($\mathcal{Q}_b$).}  As in the figures of the main text, we assumed here that the dimension is $n=2$,  $\Oc_b=(\beta,0)$, $\Oc_s=(-\beta,0)$ with $\beta=1/2$. The functions  $b(\x)$ and $s(\x)$ are defined by \eqref{eq:asym}, with $b_1(\x)=b_1(x_1,x_2)=\exp\left[-(x_1-\beta)^2/(2 \, \sigma_{x_1}^2)-x_2^2/(2 \, \sigma_{x_2}^2)\right]$, $\sigma_{x_1}^2=\sigma_{x_2}^2=1/10$, $b_0=0.7$ and $D=1/4000$. The vertical dotted line represents the threshold $\gamma^*$ obtained by solving \eqref{eq:gamma}. We obtained the numerical value of $\xb_1(t)$ at finite times by solving the PDE  ($\mathcal{Q}_b$) with a method of lines, as in the figures of the main text, with an initial condition concentrated at   $\x_0=(0,-0.3)$. To compute the limit value of $\xb_1(t)$ at $t=+\infty$, we solved the eigenvalue problem \eqref{eq:statio2} with Comsol Multiphysics eigenvalue solver.}
\label{fig:gamma}
\end{figure} 	
	
\end{document}